\documentclass[11pt]{amsart}
\usepackage{amssymb, amsmath, amsthm}
\usepackage{amsrefs}
\usepackage{graphicx}
\usepackage[final]{hyperref}
\usepackage[a4paper, centering]{geometry}

\usepackage{color}
\usepackage{graphicx}

\newtheorem{theorem}{Theorem}
\newtheorem{proposition}[theorem]{Proposition}
\newtheorem{lemma}[theorem]{Lemma}

\theoremstyle{definition}
\newtheorem{definition}[theorem]{Definition}
\theoremstyle{remark}
\newtheorem{remark}[theorem]{Remark}
\newtheorem{example}[theorem]{Example}
\def\R{\mathbb{R}}

\def\l{\lambda}

\DeclareMathOperator{\ran}{ran}

\begin{document}

\title[]%
{On a geometric combination of functions\\
related to Pr\'ekopa-Leindler inequality }%

\author[G.~Crasta, I.~Fragal\`a]{Graziano Crasta,  Ilaria Fragal\`a}

\address[Graziano Crasta]{Dipartimento di Matematica ``G.\ Castelnuovo'',
Sapienza University of Rome\\
P.le A.\ Moro 5 -- 00185 Roma (Italy)}
\email{graziano.crasta@uniroma1.it}

\address[Ilaria Fragal\`a]{
Dipartimento di Matematica, Politecnico\\
Piazza Leonardo da Vinci, 32 --20133 Milano (Italy)
}
\email{ilaria.fragala@polimi.it}



\date{April 25, 2022}

\begin{abstract} We introduce 
a new operation  between nonnegative integrable functions on $\R ^n$, that we call {\it geometric combination}; it is obtained via a mass transportation approach, playing with inverse distribution functions. The main feature of this operation is that the Lebesgue integral of the geometric combination equals the geometric mean of the two separate integrals;  as a natural consequence, we derive a new  functional inequality of Pr\'ekopa-Leindler type.  
When applied to the characteristic functions of two measurable sets,  their geometric combination provides  a set  whose volume equals the geometric mean of the two separate volumes.  
In the framework  of convex bodies, by comparing the  geometric combination with the $0$-sum, we get an alternative proof of the log-Brunn-Minkowski inequality for  unconditional convex bodies  and for convex bodies with $n$ symmetries. \end{abstract}

\maketitle

\section{Introduction}
The classical Pr\'ekopa-Leindler inequality \cite{Le,Pr1,Pr2,Pr3}  states that, given two functions $f, g \in L ^ 1 (\R ^n; \R _+)$ and a parameter $\lambda \in [0,1]$,
for any measurable function $h: \R ^n \to  \R _+$ which satisfies 
\begin{equation}\label{puntuale} h \big ((1- \lambda ) x + \lambda y \big ) \geq f(x) ^ {1 - \lambda} g ( y ) ^ \lambda \qquad \forall x, y \in \R ^n\ ,
\end{equation}
it holds that
\begin{equation}\label{f:PLclassic}
\int _{\R ^n} h \geq \Big ( \int _{\R ^n} f   \Big ) ^ {1- \lambda}
\Big ( \int _{\R ^n} g   \Big ) ^ {\lambda}  
\end{equation}
(see  also  \cite{Bar,BoCoFr,BrLi,Ga,Mar,Vi}  for several
extensions and applications). 
Inequality \eqref{f:PLclassic} is 
commonly considered as a  ``functional form'' of  the Brunn-Minkowski inequality for the $n$-dimensional volume. 
Indeed, when $f$ and $g$ are the characteristic functions of two measurable sets $K$ and $L$ in $\R^n$, the smallest function $h$ satisfying inequality \eqref{puntuale} (named $\lambda$-supremal convolution of $f$ and $g$) agrees with the characteristic function of the set $( 1-\lambda) K + \lambda L$.  Thus \eqref{f:PLclassic} yields the multiplicative form of Brunn-Minkowski inequality 
\begin{equation}\label{f:BM} |( 1-\lambda) K + \lambda L |\geq |K| ^ { 1- \l} |L| ^ \l \,,
\end{equation}
which is easily seen to be equivalent to the additive form, namely to the $(1/n)$-concavity of volume under Minkowski addition.

It is well-know that  equality in 
\eqref{f:PLclassic} occurs if and only if $f (x) = g ( x+b)$ for a log-concave function $g$ and a constant vector $b$ (see \cite{dubuc}), and in \eqref{f:BM} if and only if  $K$ and $L$ are homothetic convex bodies \cite{Sch}. 

When these conditions are far from being satisfied, estimates \eqref{f:PLclassic}  and \eqref{f:BM} may be very rough. This is one of the motivations for the investigation on one hand of quantitative versions of the inequalities \cite{BB10,FMP09,  BB11, BF14, BD21, XiLeng}, and on the other hand of possible  different operations between functions (in the analytic framework) or sets (in the geometric one), still allowing to bound from below respectively the Lebesgue integral or the volume by the geometric mean of the corresponding quantities. 

In the functional setting,  a new inequality of  Pr\'ekopa-Leindler type  has been recently obtained in  \cite{ArFlSe}: it is based on the idea of replacing 
the usual supremal convolution by a kind of geometric supremal convolution,  and still implies the Brunn-Minkowski inequality. 

In the geometric setting,  
a closely related inequality which is actually widely open and is among the most relevant questions under study in Convex Geometry is the log-Brunn-Minkowski inequality: stated within the class of centrally symmetric convex sets, the conjecture reads
\begin{equation}\label{f:logBM}|( 1-\lambda) \cdot K +_0 \lambda \cdot  L |\geq |K| ^ { 1- \l} |L| ^ \l \,,
\end{equation}
 where, denoting by $h _K$ and $h _L$ the support functions of $K$ and $L$, 
 $$( 1-\lambda) \cdot K +_0 \lambda \cdot  L := \Big \{ x \in \R ^n \ :\ x \cdot \xi \leq h _ K ( \xi) ^ { 1- \l} h _ L (\xi) ^ {\l} \quad \forall \xi \in S ^ { n-1} \Big \} \,.$$
Since the $0$-Minkowski combination $( 1-\lambda) \cdot K +_0 \lambda \cdot  L$ is contained into $( 1-\lambda) K + \lambda L$, inequality \eqref{f:logBM} is clearly a strengthening of \eqref{f:BM} (whose relevance is also related to the uniqueness of the convex body with a prescribed cone volume measure, see \cite{BLYZ}). 
Up to now, the  log-Brunn-Minkowski conjecture has been proved just in some special cases, including: planar bodies \cite{BLYZ}, unconditional bodies \cite{Boll, cordero, Sar}, bodies with symmetries \cite{BoroKala}, complex bodies \cite{Rot}; local versions have been studied in \cite{KolMil, ColLivMar, ColLiv}.

To the best of our knowledge, no functional version of \eqref{f:logBM} has been proposed up to now. There is just a  related functional inequality, 
namely the so-called multiplicative Pr\'ekopa-Leindler inequality for functions on $\R ^n _+ := ( \R _+ ) ^ n$, which can be easily deduced from the classical one via an exponential change of variables: given  $f, g \in L ^ 1 (\R ^n_+; \R _+)$ and a parameter $\lambda \in (0,1)$,
for any measurable function $h: \R ^n_+ \to  \R _+$ which satisfies
\begin{equation}\label{puntuale2} h \big ( x_ 1 ^ { 1 - \l} y _ 1 ^ \l , \dots, x_ n ^ { 1 - \l} y _ n ^ \l   \big ) \geq f(x) ^ {1 - \lambda} g ( y ) ^ \lambda \qquad \forall x, y \in \R ^n _+\ ,
\end{equation}
it holds that
\begin{equation}\label{f:PLmultiplicative}
\int _{\R ^n_+} h \geq \Big ( \int _{\R ^n_+} f   \Big ) ^ {1- \lambda}
\Big ( \int _{\R ^n_+} g   \Big ) ^ {\lambda} \, 
\end{equation}
(see \cite{uhrin, ballLN, borellLN}).
When $f$ and $g$ are the  characteristic functions of two unconditional sets $K$ and $L$,  the smallest function $h$ satisfying \eqref{puntuale2} is the characteristic function of the  product body 
$$K ^ { 1- \l} \cdot  L ^ \l:= \Big \{ z \in \R ^n \ :\ \forall i = 1, \dots, n, \ |z_i| = |x_i| ^ { 1- \l} |y_i| ^ {\l} \text{ for some } x \in K\, , \ y \in L \Big \}\,.$$ 
Hence \eqref{f:PLmultiplicative} gives $|K ^ { 1- \l} \cdot  L ^ \l | \geq |K| ^ { 1- \l} | L | ^ \l$, and  the log-Brunn-Minkowski inequality readily follows since the product body is contained into the $0$-sum.

In this paper we present a different construction: it stems in the functional analytic setting, 
where it yields a new inequality of Pr\'ekopa-Leindler type, 
 and reflects in the geometric one, where it yields a new concept of geometric mean of sets. 

Given two nonnegative integrable functions $f$ and $g$ on $\R ^n$, and a parameter $\lambda \in [0, 1]$, we introduce a new function $f \star_\l g$, that we call {\it geometric combination of $f$ and $g$} (in proportion $\lambda$), whose   Lebesgue  integral 
is equal to  the geometric mean of the integrals of $f$ and $g$.  
As a straightforward natural consequence, in order to have \eqref{f:PLclassic}, it is sufficient that $h$ is minorated almost everywhere by $f \star_\l g$.

The reason why we can handle functions defined on the whole space $\R ^n$, and not merely on  $\R ^ n _+$ as in the multiplicative Pr\'ekopa-Leindler inequality,   is that 
the construction of  $f \star_\l g$ does not involve the
geometric  mean of the variable's components  appearing in~\eqref{puntuale2}, but rather
the geometric mean of intrinsically positive quantities associated with $f$ and $g$. 
In one space dimension, such positive quantities are precisely the absolutely continuous parts of the  derivatives of the  inverse distribution functions of $f$ and $g$; in higher dimensions, the same procedure can be iterated by arguing along a prescribed family of linearly independent directions. 
In fact the proof strategy  is essentially one-dimensional, and  is 
of mass transportation type: it
incorporates the use of distribution functions originally due to Barthe (see \cite[Thm.\ 2.13]{ledoux}) with the construction of the Knothe map  (\cite{knothe}, see also \cite[p.372]{Sch}).

Moving attention from functions to sets, 
when $f$ and $g$ are the  characteristic functions of two measurable sets $K$ and $L$,  their geometric combination
agrees with the characteristic function of a measurable set, denoted by $K \star_\l L$, such that
\begin{equation}\label{f:keyeq}
| K \star_\l L | = |K| ^ { 1 - \l} |L| ^ \l \,.
\end{equation}

To the best of our knowledge, 
this way of ``geometrically combining'' two sets so that the equality \eqref{f:keyeq}  holds, is
completely new.  
Actually, several attempts exist in the literature to define some notion of geometric mean of sets, in particular of convex bodies. Besides the $0$-sum mentioned above, let us recall the dual version of the $0$-sum considered by Saroglou \cite{Sar2},  
 the classical notion of complex interpolation studied in Banach geometry \cite{Ber},  the PDE approach introduced by Cordero-Erausquin and Klartag \cite{CorKla}, the non-standard construction proposed by V.~Milman and Rotem \cite{MilRot}.  A more detailed description about each of these constructions, along with additional references, can be found in \cite{MilRot}.

While in all these cases a major concern is getting a volume estimate for  the geometric mean body,  
from this point of view   the behavior of our geometric combination is of striking simplicity, as the equality \eqref{f:keyeq} holds.  Thus, it is natural to wonder about possible relationships with the log-sum, and particularly with the log-Brunn-Minkowski conjecture. 

In this direction we are able to show that, 
 when $K$ and $L$ are unconditional convex bodies, their geometric combination $K \star_\l L$   with respect to the coordinate axes contains their $0$-combination. 
The same inclusion occurs for the class 
of convex bodies with $n$ symmetries considered in \cite{BoroKala, barcor, barfra},  provided one works with the 
natural family of directions suggested by the shapes of $K$ and $L$. 
Thus, both for unconditional bodies and for bodies with $n$ symmetries, we obtain an alternative proof of the log-Brunn-Minkowski inequality. 

Getting farther-reaching implications of our approach in the log-Brunn-Minkowski conjecture for wider classes of convex bodies remains an intriguing open question:  though  in principle  we are not fatally limited to deal with sets with special symmetries, in order to handle arbitrary sets the main difficulty seems to understand how to choose the  directions and the primitives involved in our construction.  

The paper is organized  as follows.  In Section \ref{sec:functional} we introduce
the geometric combination of functions and prove its integral property, for the sake of clarity first in one dimension (see Theorem \ref{t:1D}) and then in $n$-dimensions (see Theorem \ref{t:nD}).  In Section \ref{sec:geometric} 
we turn attention to the geometric combination of centrally symmetric convex sets, and we show that this new  operation  seems to satisfy some good properties: 
in  Section \ref{sec:convexity}  we establish the convexity preserving property (in dimension $n = 1, 2$), and we exhibit some explicit examples of geometric combinations; 
in Sections \ref{sec:unconditional}  and \ref{sec:symmetries} we deal  respectively with unconditional convex bodies  and convex bodies with $n$ symmetries and we show that, in such classes, the comparison with  the $0$-sum yields, via Theorem \ref{t:nD}, an alternative proof of the log-Brunn-Minkowski inequality. 
Finally, in Section \ref{sec:open}, we give a short list of related open problems.

\bigskip{\bf Acknowledgments.} 
We are grateful to Andrea Colesanti, Galyna Livshyts, and Emanuel Milman for some stimulating discussions
and helpful comments.

\section{Geometric combination of functions}\label{sec:functional}

\subsection{The one dimensional case}

\begin{definition}
Let $f$ be a nonnegative, integrable function of one real variable,  with  strictly positive integral.  
The \textit{inverse distribution function (shortly i.d.~function) of $f$}  is the 
generalized inverse  of
the absolutely continuous non-decreasing function
\begin{equation}\label{f:F}
F(x) := \frac{1}{\int_{\R} f} \int_{-\infty}^x f(t)\, dt\,,
\qquad x\in\R\, , 
\end{equation} 
namely 
\[
u (t):= \inf \Big \{ s \in \R \ :\ \int _{- \infty} ^ s f (x) \, dx > t \int _ \R f \Big \}\,, \qquad t \in (0, 1)\,.
\] 
\end{definition}

\begin{lemma} \label{l:idf} 
Let $f$ and $F$ be as in the above definition, and let $u$ be the i.d.~function of $f$. Then: 
 \begin{itemize}
\item[(i)] $u$ is finite valued, right continuous and strictly increasing (possibly unbounded)  in $(0,1)$;

\smallskip
\item[(ii)] $F(u(t)) = t$  for every  $t \in (0,1)$; 

\smallskip
\item[(iii)] For $\mathcal L^ 1$-a.e.\  $t\in (0,1)$, $u$ is differentiable at $t$,
$F$ is differentiable at $u(t)$, and it holds that
$F'(u(t))\, u'(t) = 1$, or equivalently
\begin{equation}\label{f:derivata}
f ( u (t)) u' ( t) = \int _\R f ( x) \, dx \qquad \text{ for $\mathcal L^ 1$-a.e. } t \in (0, 1)\,.
\end{equation}
In particular, we have that $f(u(t)) > 0$ and $u'(t) > 0$ for $\mathcal L^ 1$-a.e.\ $t\in (0,1)$.
\end{itemize} 
\end{lemma} 

\begin{proof} 
Statements (i)-(ii) follow from the basic properties of generalized inverse functions, see for instance Proposition~1 of \cite{EmHo}. 
The proof of statement (iii) can be achieved as follows.
Since $F$ is (absolutely) continuous and non-decreasing,
for every $t\in (0,1)$ the level set $\{F = t\}$ is a closed interval,
and the family $\{t_j\}\subset (0,1)$ of levels such that
$I_j := \{F = t_j\}$ has positive length is at most countable.
We can decompose $\R$ as the disjoint union $\R =  \big( \bigcup_j I_j \big )  \cup L  \cup N$, where $L$ is the set of Lebesgue points of $f$ in $\R \setminus \bigcup_j I_j$.  
Since $|F(N) | = 0$ and $\big |F\big ( \bigcup_j I_j\big )\big | = 0$,   
we have that  $F (L)$ has full measure in $(0, 1)$; then,  
letting $M\subset F(L)$ be the set of points of differentiability of $u$
in $F(L)$, also $M$ has full measure in $(0,1)$.
If $t\in M$, then $u$ is differentiable at $t$
and $u(t)$ is a Lebesgue point of $f$, so that
$F$ is differentiable at $u(t)$; the relation
$F'(u(t))\, u'(t) = 1$ is now obtained from (ii), 
and \eqref{f:derivata} follows as a direct consequence of the above analysis.
\end{proof}

\begin{remark}\label{r:Cantor} We warn that, in general, the distributional derivative of an i.d.~function may contain  jump and/or Cantor parts as in the two examples hereafter. 

\smallskip
(i) Let 
\begin{equation}\label{f:f1} 
f (x)= (1/2) \chi _{ [0,1] \cup [2, 3]}(x) \,,
\end{equation} 
where $\chi_A$ denotes the characteristic function of a set $A$.
The i.d.~function $u (t)$ equals $2t$ for $t \in [0, 1/2)$ and $2t+1$ for $t \in [1/2, 1]$. 

\smallskip
(ii) Let 
\begin{equation}\label{f:f2} f (x) = F' (x) \, , \text{ with } F (x) = \begin{cases}
0 & \text{ if } x \leq 0 \,,
\\
u ^ { -1} ( x) & \text{ if } 0 < x < 2 \,,
\\
1 & \text{ if } x \geq 2 \,, 
\end{cases}
\end{equation} 
where $u ^ {-1} $ is the inverse of the function $u:(0, 1) \to ( 0, 2)$,  $u ( t):  = t + C ( t) $, $C$ being the Cantor function. 
By construction, the i.d.~function of $f$ is precisely the function $u ( t)$. 
\end{remark}

\begin{definition} Let $u$ and $v$ be the i.d.~functions of two 
nonnegative, integrable functions  with  strictly positive integrals on $\R$, and let $\lambda \in [0,1]$.   We call 
a {\it geometric primitive of $(u, v) $ in proportion $\lambda$} any primitive of $(u') ^ { 1- \l} (v' ) ^ \l$, i.e.\ any function   of the form
\begin{equation*}
w _\l ( t) = \int _ {1/2} ^ t u' ( s) ^ {1- \l}  v' ( s) ^ { \l} \, ds + c \qquad \text {with } c \in \R \,.
\end{equation*}
In case $c= 0$ we shall refer to the geometric primitive $w _\l$ as the {\it standard} one. 
\end{definition}

\begin{remark}\label{r:AC} 
The above definition is well-posed and it holds that \begin{itemize}
\item[(i)]  $w _\l \in AC_{loc}(0,1)$ with $w _\l '(t) > 0 \text{ for } \mathcal L ^ 1\text{-a.e.}\ t\in (0,1)$; 
\smallskip 

\item[(ii)] $w_\l$  admits a classical inverse, defined on the interval
$\text{ran}\,  w _\l := \{w_\l (t)\colon t\in (0,1)\}$. 
\end{itemize}
Indeed,  from Lemma \ref{l:idf}   we have that $u', v'$ are strictly positive  $\mathcal L ^ 1\text{-a.e.}$ and belong to $L^1_{loc}(0,1)$. Hence, as a consequence  of H\"older's inequality,  also $(u')^{1-\l} (v')^{\l} $ is in $L^1_{loc}(0,1)$. 
This yields claim (i), which in turn implies (ii). 
\end{remark}

\begin{definition}\label{def:1conv} 
Let $f,g: \R \to \R$ be nonnegative, integrable functions having strictly positive integrals, with i.d.~functions $u,v$  respectively. Let $\l \in [0,1]$, and let 
$w_\l$ be a  geometric primitive  of $(u, v)$ in proportion $\lambda$.  We call the function 
\begin{equation}\label{f:1conv}
f \star_\l g  (x) := 
\begin{cases} f ( u (t)) ^ { 1- \l} g ( v (t)) ^ { \l}  & \text{ if }  x = w_\l ( t), \  t \in (0, 1)\,, 
\\   \noalign{\medskip} 
0 & \text{ otherwise ,} 
\end{cases}
\end{equation}
a \textit{geometric combination of $(f, g)$ in proportion $\lambda$}.  

In case the geometric primitive $w _\l$ in \eqref{f:1conv} is chosen as  standard one, we shall refer to 
$f \star_\l g$  as to the {\it standard} geometric combination of $(f, g)$ in proportion $\lambda$. 
\end{definition}

\begin{remark}  \label{r:well} 
Thanks to Remark \ref{r:AC},  the function  $f \star_\l g $  is well-defined and, denoting by 
$\ran w _\l$ the image of $w _\l$, it
 can be equivalently written as 
\begin{equation}\label{f:recast} 
f \star_\l g  (x) = 
\begin{cases}
f ( u (w_\l  ^ { -1} (x))) ^ { 1- \l} g ( v (w _\l ^ { -1} (x))) ^ { \l} 
&\text{if}\ x\in\ran w _\l\,,
\\
0 &\text{otherwise.}
\end{cases}
\end{equation}
Clearly, the above expression identifies uniquely $f \star _\l g$ up to a translation in the variable $x$, depending on the choice of the geometric primitive $w _ \l$. 
\end{remark}

\begin{theorem}\label{t:1D}
Let $f,g: \R \to \R$ be nonnegative, integrable functions having strictly positive integrals, let $\l \in [0, 1]$, and let 
$f \star_\l g$  be a geometric combination of $(f, g)$ in proportion $\lambda$ according to Definition \ref{def:1conv}.
Then $f \star_\l g$ is measurable and satisfies 
\begin{equation*}
\int _\R f  \star_\l g  ( x) \, dx  = \Big (\int _\R f ( x) \, dx \Big )^ { 1 - \l}  \Big (\int _\R g ( x) \, dx  \Big )^ {\l}  \,.
\end{equation*} 
\end{theorem}

\begin{proof}  
In view of \eqref{f:recast}, to show that  $ f \star_\l g$ is measurable, since both $f$ and $g$ are measurable, it is enough to show the following Lusin property (see e.g.~\cite{GSS12}): if $N$ is a set of measure zero, each of the two sets 
$(u \circ w _\l  ^{-1}) ^ { -1} (N)$ and $(v \circ w _\l^{-1}) ^ { -1} (N)$ has measure zero. Focusing for instance on the first one, we have 
$$(u \circ w_\l ^{-1}) ^ { -1}(N)  = w_\l \circ u ^ { -1}(N) =w_\l \circ F(N)\,$$
with $F$ as in \eqref{f:F}. Since $w _\l\in AC _{loc} (0, 1) $ and $F \in AC ( \R)$, $w _\l \circ F (N)$ has measure zero.  

We can now use the change of variable $x = w  _\l ( t)$ 
(see \cite[Theorem~13.32]{Yeh})
to write
\[
\begin{array}{ll}\displaystyle  \int _\R  f \star_\l g  ( x) \, dx    &  \displaystyle =
\int _ {\text{ran}\,  w_\l }f \star_\l g  ( x) \, dx    =
  \int _0 ^ 1 f \star_\l g ( w _\l ( t) ) w _\l' ( t) \, dt  \\ 
  \noalign{\bigskip}
& \displaystyle  =  \int _0 ^ 1 
 f ( u (t))^ { 1- \l}   g ( v (t)) ^ { \l}   u' ( t)^  { 1- \l}  v' ( t) ^ \l \, dt \,.
 \end{array}
\]
Finally, using \eqref{f:derivata}, we obtain 
\[
\int _0 ^ 1 
 f ( u (t))^ { 1- \l}   g ( v (t)) ^ { \l}   u' ( t)^  { 1- \l}  v' ( t) ^ \l \, dt 
 =  \Big (\int _\R f ( x) \, dx \Big )^ { 1 - \l}  \Big (\int _\R g ( x) \, dx  \Big )^ {\l}  \,.
\qedhere
\]
\end{proof}

\begin{remark}
An immediate consequence of Theorem \ref{t:1D} is the following Pr\'ekopa-Leindler type result: 
under the assumptions of Theorem \ref{t:1D} on $f$ and $g$, if $h$ is any integrable function  such that
\begin{equation}\label{f:hyp1D}
h ( x) \geq    f \star_\l g  (x)   \qquad \text{ for } \text{a.e. } x \in \R \,, 
\end{equation}
we have
 \begin{equation}\label{f:thesis1D}
 \int _\R h ( x) \, dx  \geq \Big (\int _\R f ( x) \, dx \Big )^ { 1 - \l}  \Big (\int _\R g ( x) \, dx  \Big )^ {\l}  \,,
 \end{equation}
and equality holds in \eqref{f:thesis1D} if and only if equality holds in \eqref{f:hyp1D}. 
\end{remark}

\begin{remark}\label{r:ac}  
Notice that, in general,   
$f\star_0 g$ and $f\star_1 g$ do not coincide necessarily with (a translation of) $f$ and $g$ respectively. 
Indeed, for $\l = 0$ and $\l = 1$, a geometric primitive $w_\lambda$ of their i.d.~functions $(u, v)$  does not necessarily coincide, up to a translation,  with $u$ and $v$ respectively.  For instance, referring to the examples given in Remark \ref{r:Cantor}: 
if $f$ is given by \eqref{f:f1} (and $g$ is arbitrary),
we have $ w _0( t) = 2( t - \frac{1}{2})$, and hence $f \star _0  g=  (1/2) \chi _{ [0,2]}$; 
if $f$ is given by \eqref{f:f2} (and $g$ is arbitrary),  we have $w_ 0 ( t) = t$ and hence $f \star _0 g = (\int _\R f)\chi _{[0,1]}$.  
\end{remark}

When the i.d.~functions of $f, g$ enjoy suitable regularity assumptions (which as we shall see occurs for the characteristic functions of convex bodies, see Remark \ref{r:interpolation}), the geometric combination $f \star _ \l g$ provides a continuous interpolation between $f$ and $g$.

\begin{proposition}\label{p:contl} Let $f$ and $g$ satisfy the same assumptions  of Theorem \ref{t:1D}. If in addition their 
i.d.~functions belong to
$AC_{\text{loc}}(0,1)$, we  have that:
\begin{itemize}
\item[(i)]  up to translations, $f\star_0 g = f$ and $f\star_1 g= g$;
\medskip
\item[(ii)]  the map
$[0,1]\ni\lambda\mapsto f\star_\l g$ is continuous in
$L^1(\R)$.
\end{itemize} 
\end{proposition}

\begin{proof}  Let   $u, v$ be the i.d.~functions of $f,g$ respectively. By assumption $u, v\in AC_{\text{loc}}(0,1)$, and thus 
statement (i) immediately follows from the equalities $w_0 = u$ and $w _ 1 = v$. 

Let us prove statement (ii). Let $(\lambda_j) \subset [0,1]$ be a sequence converging
to $\overline{\lambda}\in [0,1]$, and let us define
\[
\psi_j(t) := f(u(t))^{1-\lambda_j} g(v(t))^{\lambda_j}\,,
\quad
\psi(t) := f(u(t))^{1-\overline{\lambda}} g(v(t))^{\overline{\lambda}}\,, 
\qquad t\in (0,1),
\]
$f\star_{\lambda_j} g = \psi_j \circ w_{\lambda_j}^{-1}$
on $I_j:=\ran w_{\lambda_j}$,
$f\star_{\overline{\lambda}} g = \psi \circ w_{\overline{\lambda}}^{-1}$
on $I := \ran w_{\overline{\lambda}}$,
all functions vanishing otherwise.

We claim that:
\begin{itemize}
\item[(a)]
$\lim_j f\star_{\lambda_j} g (x) = f\star_{\lambda} g (x)$,
for a.e.\ $x\in I$. 

\item[(b)] 
$\displaystyle \lim_{j\to+\infty}\int_{I} f\star_{\lambda_j} g =
\int_{I} f\star_{\overline{\lambda}} g,
\qquad
\lim_{j\to+\infty}\int_{\R\setminus I} f\star_{\lambda_j} g =0.
$
\end{itemize}
Before proving this claim, let us show how the convergence
of $f\star_{\lambda_j} g$ to
$f\star_{\overline{\lambda}} g$ in $\L^1(\R)$ follows
from (a) and (b).

Specifically, from (a), the first equality in (b) 
and \cite[Theorem~16.28]{Yeh}, it follows that
$f\star_{\lambda_j} g$ converges to
$f\star_{\overline{\lambda}} g$ in $L^1(I)$.
On the other hand, from the second equality in (b) it follows that
$f\star_{\lambda_j} g$ converges to
$0 = f\star_{\overline{\lambda}} g$ in $L^1(\R\setminus I)$.

\noindent
\textit{Proof of (a).}
As a first step, let us show that if $x\in I$, then $x\in I_j$ for $j$
large enough.
Specifically, it is enough to observe that, for every $\lambda\in [0,1]$,
\[
\ran w_\lambda = (-a_\lambda, b_\lambda),
\quad\text{with}\quad
a_\lambda := \int_{0}^{1/2} (u')^{1-\lambda} (v')^{\lambda},
\quad
b_\lambda := \int_{1/2}^1 (u')^{1-\lambda} (v')^{\lambda},
\]
and that, by Fatou's lemma,
\[
a_{\overline{\lambda}}\leq \liminf_{j\to+\infty} a_{\lambda_j}\,,
\qquad
b_{\overline{\lambda}} \leq \liminf_{j\to+\infty} b_{\lambda_j}\,.
\]

Since, for a.e.\ $t\in (0,1)$ we have that 
$\lim_j \psi_j(t) = \psi(t)$,
(a) will follow if we show that, for every $x\in I$,
$\lim_j w_{\lambda_j}^{-1}(x) = w_{\overline{\lambda}}^{-1}(x)$.
Given $x\in I$, for $j$ large enough let 
$t_j := w_{\lambda_j}^{-1}(x)$.
Let $(t_{j_k})$ be a subsequence converging to some~$\tau \in [0, 1]$.

If $\tau \in (0,1)$, 
it holds that
\[
x = w_{\lambda_j}(t_{j_k})
= \int_{1/2}^{t_{j_k}} (u')^{1-\lambda_{j_k}} (v')^{\lambda_{j_k}}
\longrightarrow
\int_{1/2}^{\tau} (u')^{1-\overline{\lambda}} (v')^{\overline{\lambda}}
= w_{\overline{\lambda}}(\tau),
\] 
where the convergence holds since 
$0 \leq  (u')^{1-\lambda_{j_k}} (v')^{\lambda_{j_k}} \leq u' + v'
\in L^1_{\text{loc}}(0,1)$.
Hence, $x = w_{\overline{\lambda}}(\tau)$, i.e.,
$\tau = w_{\overline{\lambda}}^{-1}(x)$.

Assume now that $\tau = 1$.
The sequence 
$q_k := (u')^{1-\lambda_{j_k}} (v')^{\lambda_{j_k}}\, \chi_{[1/2, t_{j_k}]}$
converges a.e.\ in $[1/2, 1]$ to
$q :=  (u')^{1-\overline{\lambda}} (v')^{\overline{\lambda}}$,
hence,
by Fatou's Lemma we have that
\begin{equation*}
b_{\overline{\lambda}} =
\int_{1/2}^1 q \leq \liminf_{k\to +\infty} \int_{1/2}^1 q_k = x\,,
\end{equation*}
in contradiction with
the assumption that $x$ belongs to the open interval 
$I 
= (-a_{\overline{\lambda}}, b_{\overline{\lambda}})$.

A similar argument shows that $\tau \neq 0$.

We have thus shown that every converging subsequence of $w_{\lambda_j}^{-1}(x)$
converges to 
$w_{\overline{\lambda}}^{-1}(x)$, hence the
whole sequence converges to $w_{\overline{\lambda}}^{-1}(x)$.

\noindent
\textit{Proof of (b).}
Since $\int_{\R} f$ and $\int_{\R} g$ are strictly positive,
by Theorem~\ref{t:1D}
it holds that
\begin{equation}\label{f:equalint}
\int_{\R} f\star_{\lambda_j} g =
\left(\int_{\R} f\right)^{1-\lambda_j}
\left(\int_{\R} g\right)^{\lambda_j}
\stackrel{j\to+\infty}{\longrightarrow}
\left(\int_{\R} f\right)^{1-\overline{\lambda}}
\left(\int_{\R} g\right)^{\overline{\lambda}}
= \int_{\R} f\star_{\overline{\lambda}} g\,.
\end{equation}
Hence, it is enough to prove only the second equality in (b).
From \eqref{f:equalint}, (a) and Fatou's Lemma we have that
\[
0\leq \limsup_{j\to +\infty} \int_{\R\setminus I} f\star_{\lambda_j} g
=
\limsup_{j\to +\infty} \left(
\int_{\R} f\star_{\lambda_j} g
- 
\int_{I} f\star_{\lambda_j} g\right)
\leq 
\int_{\R} f\star_{\overline{\lambda}} g
-
\int_{I} f\star_{\overline{\lambda}} g = 0,
\]
so that the second equality in (b) follows.
\end{proof}

\subsection{The $n$-dimensional case}

Let $(z_1 , \dots,  z_n)$  be a family of linearly independent vectors in $\R^n$ 
(a typical choice will be the canonical basis $(e_1, \dots, e _n)$ of $\R ^n$, cf.~Section \ref{sec:geometric}).  
For simplicity of notation,  we do not indicate the dependence of our construction on the family $(z_1 , \dots,  z_n)$, as it  will remain fixed throughout this section.

\begin{definition}\label{d:field}
Given a nonnegative integrable function $f$ on $\R^n$ with strictly positive integral, 
we call 
$(z_1, \dots, z _n)$-{\it inverse distribution field (shortly i.d.~field) of $f$}  the vector field 
$$ U ( t)= u_1 ( t_1) z_1+  u _ 2 ( t_1, t_2) z_ 2 + u _ 3 (t_1, t_2, t_3) z_3 + \dots + u _n ( t_1, t_2, \dots, t _n)z _n \,,$$  
 defined for $\mathcal L ^ n$-a.e.\ $t = (t_1, \dots, t_n) \in (0,1)^n$  as follows:
 \begin{itemize}
\item $t_1 \mapsto u _ 1(t_1)$ is the i.d.~function of the map
$$x _ 1 \mapsto  \int _{ \R ^ { n-1}}  f  ( x_1z _ 1 +  x_ 2 z _ 2 + \dots + x _n z _n)  \, dx_2  \dots  d x _n \,;$$ 
\item  for $\mathcal L ^ 1$-a.e.\ $t _1\in (0, 1)$, $t_2 \mapsto u _ 2 ( t_1 ,t_2)$ is the i.d.~function of the map 
$$x _ 2 \mapsto \int _{ \R ^ { n-2} } f (u _1 ( t_1) z _ 1 + x_2 z _ 2 + \dots + x _n z _n)  \, d x _ 3  \dots \, d x _n\,;$$  
\item $\dots$
\item for $\mathcal L ^ { n-1}$-a.e.\ $(t _1, \dots , t _ { n-1})\in (0, 1) ^ { n-1}$, $t _n \mapsto u _n ( t_1, \dots, t _n)$ is the i.d.~function of the map  
$$x _ n \mapsto  f (u _1 ( t_1) z _ 1 + u _ 2 ( t_1, t_2) z _ 2 + \dots + u _ {n-1} ( t _1, t_2 , \dots, t _ {n-1} ) z _ {n-1}  + x _n z _n)\,.$$  
\end{itemize} 
\end{definition}

\begin{remark} \label{r:produttoria} The above definition is well-posed and leads to a $n$-dimensional analogue of the identity  \eqref{f:derivata}. Indeed: 
\begin{itemize}
\item   the i.d.~function $t _ 1 \mapsto u _ 1 ( t_1)$ is well-defined  (because by assumption $f$ has strictly positive integral), and for  $\mathcal L ^1$-a.e.\ $t _ 1 \in (0,1)$ 
it satisfies    the equality  
\[
\begin{split}
& u' _ 1 ( t_1) \int _{ \R ^ { n-1}}  f  ( u_1 (t_1) z _ 1 +  x_ 2 z _ 2 + \dots + x _n z _n) \, dx_2  \dots  d x _n  
\\ & = \int_{\R^n} f(x_1 z_1 + \cdots + x_n z_n)\, dx_1\ldots dx_n  
\quad \Big ( = \frac{1}{|z _ 1\wedge \dots \wedge z _n |} \int _ { \R ^ n }  f \Big )  \,;
\end{split}
\]

\item for  $\mathcal L ^1$-a.e.\ $t _ 1 \in (0,1)$,  the i.d.~function $t _ 2 \mapsto u _ 1 ( t_1, t_2)$ is well-defined (because by the previous item the integral at the r.h.s.\ of the equality below is strictly positive),  and  for $\mathcal L ^1$-a.e. $t _ 2 \in (0,1)$  it satisfies
\[
\begin{split}
& \frac{\partial u _ 2} {\partial t _ 2} ( t_1, t _ 2) \int _{ \R ^ { n-2} } f (u _1 ( t_1) z _ 1 + u_2 ( t_1, t_2) z _ 2 + \dots + x _n z _n) \,  d x _ 3  \dots \, d x _n  
\\ & 
=  \int _{ \R ^ { n-1}}  f  ( u_1 ( t_1) z _ 1 +  x_ 2 z _ 2 + \dots + x _n z _n) \, dx_2  \dots  d x _n
\end{split}
\]

\item $\dots$

\item   for  $\mathcal L ^{n-1}$-a.e.\ $( t_1, \dots , t _ {n-1}) \in (0,1) ^ { n-1}$,  the i.d.~function $t _ n \mapsto u _ n ( t_1, \dots,  t_{n-1}, t _n)$ is well-defined (because by the previous items the integral at the r.h.s. of the equality below is strictly positive),  and for $\mathcal L ^1$-a.e. $t _ n \in (0,1)$  it satisfies:
\[
\begin{split}
& \frac{\partial u _ n}  {\partial t _ n} ( t_1, t _ 2, \dots, t _ n ) f ( u _ 1 ( t _ 1) z _ 1 + u _ 2 ( t_1, t _2) z _ 2 + \dots + u _ n ( t _1, t_2,  \dots, t _n) z _ n)   
\\ & 
= \int  _ \R f ( u _ 1 ( t _ 1) z _ 1 + u _ 2 ( t_1, t _2) z _ 2 +  \dots + x _n z _n ) \, d x _n \,. 
\end{split}
\]
\end{itemize}

Multiplying side by side the above equalities, we get by analogy with \eqref{f:derivata} the identity 
\begin{equation}\label{f:product}
f ( U ( t))\, \prod _{i= 1} ^n \frac{\partial u _ i} {\partial t _i} = \frac{1}{|z _ 1\wedge \dots \wedge z _n |}\int _ { \R ^ n} f    
\qquad \text {for $\mathcal L ^n$-a.e.\ } t \in ( 0,1) ^ n \,.
\end{equation}  
\end{remark}

\begin{definition}  
Let $U$ and $V$ be the $(z_1 , \dots,  z_n)$-i.d.~fields of two nonnegative integrable function $f$ and $g$
with strictly positive integrals on $\R^n$.
We call $(z_1, \dots,  z_n)$-{\it geometric potential of $(U, V)$ in proportion $\lambda$} any vector field 
$$W _\l (t)=  w_{\l, 1} ( t_1) z _ 1+ w _{\l, 2} ( t_1, t_2) z _ 2 + \dots + w _{\l,n } ( t_1, \dots t _n ) z _n$$ such that: 
 \begin{itemize}
\item $t _1 \mapsto w _ {\l,1} (t_1)$  is a geometric primitive in proportion $\l$ of $t _ 1 \mapsto u _ 1 ( t_1)$, $t _ 1 \mapsto v_ 1 ( t_1)$;
\smallskip
\item  for $\mathcal L ^ 1$-a.e.\ $t_1 \in (0, 1)$,  
$t_2 \mapsto w_ {\l,2} ( t_1, t_2)$  is a geometric primitive in proportion $\l$ of $t _ 2 \mapsto u _2 ( t_1, t_2)$, $t _ 2 \mapsto v_ 2 ( t_1, t_2)$;
\smallskip
\item $\dots$
\smallskip
\item  for $\mathcal L ^ {n-1}$-a.e.\ $(t_1, \dots, t _ {n-1})  \in (0, 1)^ { n-1}$,  
$t_n \mapsto w_ {\l,n} ( t_1,  \dots, t _ n)$  is a geometric primitive in proportion $\l$ of $t _ n \mapsto u_n ( t_1,  \dots, t _ n) $, $t _ n \mapsto v_n ( t_1,  \dots, t _ n)$.
\end{itemize} 
In case all the geometric primitives $w _{\l, i}$ are the standard ones, 
 namely when
 \begin{equation*}
w_1 ( 1/2) = 0\, , \quad w _ 2 ( t _1, 1/2) = 0\,, \quad \dots ,\quad  w _n ( t_1, \dots, t _ {n-1}, 1/2) = 0\,,
\end{equation*} 
we shall refer to the geometric potential $W _\l$ as to the {\it standard} one.
\end{definition}

\begin{remark}\label{r:acloc} 
By analogy with one-dimensional case, we have that
\begin{itemize} 
\item[(i)] 
for $\mathcal L ^ { i-1}$-a.e.\ $(t_1, \dots, t _{i-1})  \in ( 0,1) ^ { i-1}$, the maps $t_i \mapsto w_{\l, i} ( t _1, \dots, t_{i-1}, t_i)$ are in  $ AC_{loc}(0,1)$,  with 
\begin{equation}\label {f:jacobians}   \frac{\partial w _{\l, i}} {\partial t _i}  =   \Big ( \frac{\partial u _ i} {\partial t _i}  \Big ) ^ { 1- \l}   \Big ( \frac{\partial v _ i} {\partial t _i}  \Big )  ^ { \l} >0 
 \qquad 
\text {for $\mathcal L ^1$-a.e.\ } t_i \in ( 0,1) \,;
\end{equation} 
\item[(ii)]
$W _\l$ admits a classical inverse defined on $\ran W _\l := \{ W _ \l ( t) \ :\ t \in (0, 1) \}$. 
\end{itemize}  
\end{remark}

\begin{definition}\label{def:nconv}
Let $f,g: \R^n \to \R$ be nonnegative, integrable functions having strictly positive integrals, with 
$(z_1,\dots, z _ n)$-i.d.~fields 
 $U,V$ respectively. 
 If $W _\l$ is a $(z_1,\dots, z _ n)$-geometric potential  of $(U, V)$ in proportion $\lambda$, we call the function 
\begin{equation}\label{f:nconv}
f \star_\l g  (x) := 
\begin{cases}
f ( U (t)) ^ { 1- \l} g ( V (t)) ^ { \l}  & \text{ if }  x = W _\l (t) \, , \  t \in (0, 1) ^ n \,,
\\ 
0 & \text{ otherwise,}
\end{cases} 
\end{equation}
\textit{a $(z_1, \dots, z_n)$-geometric combination of $(f, g)$ in proportion $\lambda$}. 
 
 In case the geometric potential $W _\l$ in \eqref{f:nconv} is chosen as the standard one, we shall refer to 
$f \star_\l g$  as to the {\it standard} geometric combination of $(f, g)$ in proportion $\lambda$.

 \end{definition}

\begin{remark} Again by analogy with the one-dimensional case,  denoting by $\ran W _\l$ the image of $W _\l$, we have 
\begin{equation}\label{f:recastndim} 
f \star_\l g  (x) = 
\begin{cases}
f ( U (W _\l^ { -1} (x))) ^ { 1- \l} g ( V (W_\l^ { -1} (x))) ^ { \l}
&\text{if}\ x\in \ran W _\l,
\\
0 & \text{otherwise.}
\end{cases}
\end{equation}
Notice that, in the present $n$-dimensional setting, 
the function $W_\l$ is identified
up to an additive constant in the first component, up to a function of $x_1$ in the second component, up to a function of $x_2$ in the third one, and so on, up to a function of $(x_1, \dots, x_{n-1})$ in the last component.
\end{remark}

\begin{theorem} \label{t:nD} 
Let $f,g: \R^n \to \R$ be nonnegative, integrable functions having strictly positive integrals. Let $\l \in [0, 1]$, and let 
$f \star_\l g$  be a $(z_1, \dots, z_n)$-geometric combination of $(f, g)$ in proportion $\lambda$ according to Definition \ref{def:nconv}.
Then $f \star_\l g$ is measurable and satisfies 
\begin{equation*}
\int _{\R ^n} f  \star_\l g  ( x) \, dx  = \Big (\int _{\R^n} f ( x) \, dx \Big )^ { 1 - \l}  \Big (\int _{\R ^n} g ( x) \, dx  \Big )^ {\l}  \,.
\end{equation*} 
\end{theorem}

\begin{proof}
  To show that  $ f \star_\l g$ is measurable,  thanks to \eqref{f:recastndim}   it is enough to show that,  if $N$ is is Lebesgue negligible, each of the two sets 
$(U \circ W _\l ^{-1}) ^ { -1} (N)$ and $(V \circ W _\l ^{-1}) ^ { -1} (N)$ has measure zero. Considering for instance the set 
$(U \circ W_\l  ^{-1}) ^ { -1}(N)$, we can write it as $W_\l  \circ U ^ { -1}(N)$. Thus we see that  it is Lebesgue negligible because $U$ is the inverse distribution
 field of the integrable function $f$, and  the components of $W _\l$ have the property that the maps $t _ i \mapsto w_{\lambda , i}$ are locally absolutely continuous
(see Remark~\ref{r:acloc}).

Let us compute the integral of $ f \star_\l g$.
Let us define the function
$h(x) := f \star_\l g(x_1 z_1 + \cdots + x_n z_n)$,
$x\in\R^n$,
so that
\begin{equation}\label{f:cvar}
\int _{\R^n}   f \star_\l g  ( x) \, dx 
= 
{|z_1 \wedge \dots \wedge z_n |}
\int_{\R^n} h(x) \, dx\,.
\end{equation}
In order to compute the integral of $h$, we perform $n$ one-dimensional changes of variable.
Proceeding as in the proof of Theorem~\ref{t:1D},
and recalling that, by Definition~\ref{def:nconv}, $h$ vanishes outside
the set
$\{(w_{\l,1}(t_1), \ldots, w_{\l, n}(t_1,\ldots, t_n))\colon 
t_i \in (0,1)\}$,
we first use the change of variable $x_1 = w_{\l, 1}(t_1)$
(see \cite[Theorem~13.32]{Yeh}), obtaining
\[
\int_{\R^n} h(x) \, dx
= \int_{(0,1)} \left(\int_{\R^{n-1}}
h(w_{\l, 1}(t_1), x_2,\ldots,x_n)\, dx_2\ldots dx_n)\, w_{\l, 1}'(t_1)\right)\, dt_1\,.
\]
Next, in the inner integral we proceed with the change of variable
$x_2 = w_{\l, 2}(t_1, t_2)$, $t_2\in (0,1)$.
Finally, after the last change of variable
$x_n = w_{\l, n}(t_1,\ldots, t_n)$, $t_n\in (0,1)$,
using the family of equalities~\eqref{f:jacobians},
and recalling the Definition~\ref{def:nconv} of
$f\star_\lambda g$,
we end up with
\[
\begin{split}
\int_{\R^n} h(x) \, dx
& = 
\int _{(0, 1) ^ n }  h(w_{\l,1}(t_1),\ldots, w_{\l,n}(t_1,\ldots, t_n))\,    
\prod _{i= 1} ^n \frac{\partial w _ {\l,i}} {\partial t _i} 
  \,  dt
\\ &
= 
\int _{(0, 1) ^ n }  h(w_{\l,1}(t_1),\ldots, w_{\l,n}(t_1,\ldots, t_n))\,    
\prod _{i= 1} ^n \Big ( \frac{\partial u _ i} {\partial t _i}  \Big ) ^ { 1- \l}  \prod _{i= 1} ^n  \Big ( \frac{\partial v _ i} {\partial t _i}  \Big )  ^ { \l}
  \,  dt 
\\ &
= 
\int _{(0, 1) ^ n }  f(U(t))^{1-\l} g(V(t))^\l\,    
\prod _{i= 1} ^n \Big ( \frac{\partial u _ i} {\partial t _i}  \Big ) ^ { 1- \l}  \prod _{i= 1} ^n  \Big ( \frac{\partial v _ i} {\partial t _i}  \Big )  ^ { \l}
  \,  dt\,. 
\end{split}
\]
Finally, from \eqref{f:cvar} and in view of \eqref{f:product},
we conclude that
\[
\int _{\R^n}   f \star_\l g  ( x) \, dx 
=
\left(\int_{\R^n}f(x)\, dx\right)^{1-\l}
\left(\int_{\R^n}g(x)\, dx\right)^{\l}\,.
\qedhere
\]
\end{proof}

\section{Geometric combination of convex bodies}\label{sec:geometric}

In  this section we focus attention on the geometric combination between the characteristic functions of two nondegenerate, centrally symmetric, convex bodies 
$K$ and $L$ in $\R ^n$.

Below, we write for simplicity $K \star_\l L$ in place of $\chi _K \star_\l \chi _L$; moreover, we refer to  the i.d.~function (or field) of the characteristic function $\chi _K$  as the i.d.~function (or field) of $K$.  
Since we are dealing with nondegenerate convex bodies,
for simplicity of notation
in the proofs sometimes we identify a convex body with its interior.

\subsection{Convexity preserving for $n=1, 2$}
\label{sec:convexity}

When $K$ is a symmetric interval, $K= \big [ - a,  a  \big ]$, its i.d.~function is given by
\begin{equation}\label{f:1idf} 
u  ( t) = a \big ( 2 t - 1 \big )
\qquad \forall t \in (0, 1 )\,.
\end{equation}
It is elementary to deduce the behavior of the geometric combination for intervals: 

\begin{proposition}\label{p:intervals} 
Let  $K$ and $L$ be 
symmetric intervals, and $\lambda \in [0,1]$. 
Then the standard geometric combination $K \star _\l L$ is itself  a symmetric interval, precisely
\[
K \star _\l L = ( 1- \l)\cdot  K +_0 \l \cdot L \,.
\]
In particular, the equality $| K \star_\l L | = |K| ^ { 1 - \l} |L| ^ \l$ given by Theorem~\ref{t:1D} is equivalent to  the one-dimensional log-Brunn-Minkowski equality. 
\end{proposition}

\begin{proof} 
If $u$ and $v$ are the i.d.~functions of 
$K = \big [ - a,  a  \big ]$ and $L = \big [- b  ,  b \big ]$, by \eqref{f:1idf}
the standard geometric primitive of $(u,v)$ in proportion $\lambda$  is given by
$$w_\l ( t) = a ^ { 1- \l} b ^ \l  \big (2 t -  1 \big )  \,.$$ 
Hence, the 
body $K \star _\l L$, which is the image of $w_\l$, is given by 
\[
K \star _\l L = \{ w _\l ( t) \, :\, t \in (0,1)\} =  \big  (- a ^ { 1- \l} b ^ \l ,   a ^ { 1- \l} b ^ \l \big  )\,.
\] 
On the other hand, since the support functions of $K$ and $L$ are given respectively by 
$h _ K ( \xi) = a |\xi|$ and $h _ L ( \xi) = b |\xi|$,  the log-Brunn-Minkowski sum $( 1- \l) K + _ 0 \l L $  is the interval with support function
$h _ K ( \xi) ^ { 1- \l } h _ L (\xi ) = a ^ { 1- \l} b ^   \l  |\xi |$,
namely  we have as well 
\[
( 1- \l) \cdot  K +_0 \l \cdot  L  =   \big  (- a ^ { 1- \l} b ^ \l , a ^ { 1- \l} b ^ \l \big  )\,.
\qedhere
\]
\end{proof}

\bigskip

When $K$ is a  centrally symmetric convex body in  $\R^2$, setting 
\[
K _ { x_1}:= \big \{ x _ 2 e _ 2 \, :\, x_ 1 e _ 1 + x _ 2 e _ 2 \in K \big \}\,,
\] 
the components $(u_1,u_2)$ of the $(e_1, e_2)$-i.d.~field of $K$ satisfy 
\begin{itemize}
\item $t_1 \mapsto u _ 1(t_1)$ is the i.d.~function of the map
$x _ 1 \mapsto  
\mathcal H ^ {1}  ( K _ {x_1})$, 
hence
$$ u _ 1 ' ( t_1) = \frac{|K| }{\mathcal H ^ {1} ( K _ {u_1(t_1)})} \qquad \text{ for $\mathcal L ^ 1$-a.e. } t _ 1\in (0, 1) \,;$$
\item  for $\mathcal L ^ 1$-a.e.\ $t _1\in (0, 1)$, $t_2 \mapsto u _ 2 ( t_1 ,t_2)$ is the i.d.~function of the map 
$x _ 2 \mapsto  
 \chi_{ K _ { u _ 1 ( t_1 )}}  (x_2)$;
hence,
\[
\frac{ \partial u _ 2}{\partial t _ 2}   (t_1, t _ 2)=  { \mathcal H ^ {1} ( K _ { u _ 1 ( t_1 )} )}   \qquad \text{ for $\mathcal L ^ 1$-a.e. } t _ 2\in (0, 1) \,.
\] 
\end{itemize} 
As a consequence, we prove below that the standard geometric combination of centrally symmetric planar convex bodies preserves convexity (and always produces an unconditional set).

\medskip 
\begin{proposition}\label{p:convexity} 
Let
$K$ and $L$ be centrally symmetric convex bodies in $\R^2$, and let $\lambda \in (0,1)$. 
Then the standard $(e_1, e_2)$-geometric combination $K \star _\l L$ is an unconditional convex body. 
\end{proposition} 

\begin{proof} 

For brevity, we are going to denote by $M$ the standard $(e_1, e_2)$-geometric combination $K \star _\l L$ (where $\l$ is fixed in $[0, 1]$). 

To see that $M$ is unconditional, recall that $M$ is the image of  the standard geometric potential $W _\l$ of the i.d.~fields $(U, V)$ of $K$ and $L$. 
In view of the above expressions of $u' _ 1 ( t_1)$ and $\frac{ \partial u _ 2}{\partial t _ 2}$ (and of their analogue for the derivatives of the components of $V$), 
we have 
$$W_{ \l}  ( t_1, t_2) = \Big ( \int _{ \frac{1}{2}}  ^ { t _ 1} \!\!\! \Big (  \frac{  |K| }{\gamma _K ( s)}  \Big )  ^ { 1 - \l}  \Big ( \frac{ |L | }{\gamma _L (s) }   \Big ) ^ \l  \, ds \, , \   
 \gamma _K ( t_1)  ^ { 1 - \l} \gamma_L ( t_1) ^ \l  \Big ( t _ 2 - \frac{1}{2} \Big )  \Big ) \quad \forall ( t_1, t _2) \in ( 0,1)^2
  \,, 
  $$ 
where we have set for brevity \begin{equation*}
 \gamma _K (t_1) :=  \mathcal H ^ 1 (K _ { u _ 1 (t_1 ) } )  \, , \quad \gamma _L (t_1) :=  \mathcal H ^ 1 (L _ { v _ 1 (t_1 ) } )  \qquad \forall t _ 1 \in (0, 1) \, .
 \end{equation*} 
 
It readily follows that the components of $W _\l$ satisfy,
for every $\delta\in (0, 1/2)$, 
\[
\begin{array}{ll} & w _ {\l,1 } (\frac{1}{2} + \delta) = -  w _{\l,1} (\frac{1}{2} - \delta)\,,\\  \noalign{\bigskip} 
& w _ {\l,2} (t_1, \frac{1}{2} + \delta) = -  w _ {\l,1} (t_1, \frac{1}{2} - \delta ) \quad \forall t _ 1 \in (0, 1) \,, 
\end{array}
\] 
which shows that $M$ is unconditional.

We now turn attention to convexity. 
We claim that the convexity of a generic centrally symmetric set $K$ is related to the concavity of the corresponding  function $\gamma _K ^2$ as follows: 
\begin{eqnarray}
&K \text {  convex} \   \Rightarrow \ 
\gamma _K ^ 2  \text { concave} & \label{f:convexity1}  
\\ \noalign{\medskip} 
& K \text{ unconditional  and } \gamma _K ^ 2  \text { concave}   \   \Rightarrow \ 
K \text { convex.}
& \label{f:convexity2} 
 \end{eqnarray}

Specifically, setting  $\psi _K (x_1):= \mathcal H ^ 1 ( K _ { x_1}) $, 
we have that: if $K$ is convex, the map $\psi _K$ is concave; viceversa, if $\psi _K$ is concave and $K$ is symmetric about the $x_1$-axis (and hence unconditional since it is assumed to be centrally symmetric), then $K$ is 
convex. 

In view of this observation,  to obtain  \eqref{f:convexity1}-\eqref{f:convexity2},  it is enough to show that
the concavity of $\psi _ K$  is equivalent to the concavity of $\gamma _ K ^ 2$. 
From the definition of the distribution function $u _1$, we have that
\[
\psi _ K  (u _ 1  (s))   u' _ 1 (s)  =   |K|\,,
\qquad\text{for a.e.}\ s\in (0,1) \,.
\]  
Since $\gamma _ K ( s) =  \psi _ K (u _ 1 ( s))$, squaring both sides and differentiating with respect to $s$  gives
\[
(\gamma^ 2 _ K ) ^ {\prime} ( s)  = 2  \psi _K ( u_1) \psi ^ {\prime} _K ( u_1) u_1 ^ {\prime}  = 2  |K| \psi_K ^ {\prime} ( u_1 ) 
\]
 which shows that  $(\gamma^ 2 _ K ) ^ {\prime} $ is nondecreasing if and only if $\psi_K ^ {\prime}$ is (recall indeed that $u _ 1$ is increasing).

Now, thanks to \eqref{f:convexity1}-\eqref{f:convexity2}, and since we have already proved that $M$ is unconditional, in order to prove that $M$ is convex, we are reduced to show that
\begin{equation}\label{f:meanconv} 
\gamma _K ^ 2  \text{ and } \gamma _L ^ 2  \text { concave}\ \Longrightarrow \gamma _M ^ 2  \text { concave} \,.
\end{equation} 
We observe that, by the definition of $M$, it holds that
$\gamma_M = \gamma _K   ^ { 1 - \l} \gamma _L  ^ \l$. 
Then the validity of the implication \eqref{f:meanconv} follows from the fact that the geometric mean of two nonnegative concave function is still concave. For the sake of completeness, we enclose the elementary proof. Denoting by $\varphi$ and $\psi$ the two functions, we have 
\[
\begin{split}
\varphi  ^ { 1 - \l } & ( ( 1 - \theta) s + \theta t )  \psi ^ \l ( ( 1 - \theta) s + \theta t )  
\\ 
& \geq  ( ( 1 - \theta)  \varphi   (s) + \theta \varphi (t) )^ { 1 - \l }    ( ( 1 - \theta) \psi (s) + \theta\psi ( t) ) ^ \l    
\\  
& \geq  ( 1 - \theta) \varphi ( s) ^ {1 - \l} \psi ( s) ^ \l + \theta  \varphi ( t) ^ {1 - \l} \psi ( t) ^ \l  \,,
\end{split}
\]
where in the last line  we have exploited the inequality $(x_ 1 + y _ 1)  ^ { 1 - \l} ( x _ 2 + y _ 2) ^ \l \geq 
x_ 1 ^ { 1 - \l} x_2 ^ \l + y_ 1 ^ { 1 - \l} y_2 ^ \l$, holding for nonnegative numbers $x_1, x_2, y_1, y _ 2$ (as it follows by adding the a.m.-g.m. inequality applied separately to the pair  $(\frac{x_1}{x_1+ y_1}, \frac{x_2}{x_2 + y _ 2})$ and to the pair
$(\frac{y_1}{x_1+ y_1}, \frac{y_2}{x_2 + y _ 2})$). 
\end{proof}

\bigskip
Below we give some explicit examples of geometric combinations in dimension $2$.

 \begin{example}[geometric combination of two rectangles]\label{e:rectangles}
 If $K= Q ( a_1, a_2)$ is the centrally symmetric rectangle with vertices at  
$(\pm a_1, \pm a_2)$, with $a_i >0$, we have $\mathcal H ^ 1 (K _ { u _ 1 ( s) }) = 2 a_2 $, so that the i.d.~field of $K$ is given by 
 $$U ( t_1, t_2) = \big (   a_1 ( 2 t _ 1 - 1)  \, , \ a_2 ( 2 t _ 2 -1) \big ) \qquad \forall (t_1, t_2) \in (0, 1) ^ 2 \,.$$ 
If we take another rectangle $L = Q ( b_1, b_2)$ with i.d.~field $V$, 
the standard $(e_1, e_2)$-geometric potential of $(U, V)$ in proportion $\lambda$  has components  
\[
W_{\l} ( t_1, t _2) = \big ( a_1^{1-\l}  b_1^{\l} ( 2 t _ 1 - 1)  \, , \  
a_2^{1-\l} b_2^{\l} ( 2 t _ 2 -1)  \big )  \,, \qquad \forall (t_1, t_2) \in (0, 1) ^ 2 \,.
\]
Hence, 
\[
Q ( a_1, a_2) \star_{\lambda } Q ( b_1, b_2) = Q ( a_1^{1-\l}  b_1^{\l}, a_2^{1-\l} b_2^{\l} )\,.
\]
 \end{example}

 \begin{example}[geometric combination of two parallelograms with parallel sides]
 By arguing as in the previous example, it is immediate to obtain that the $(z_1, z_2)$-geometric combination in proportion $\lambda$ of two parallelograms with sides parallel to the vectors $(z_1, z_2)$, of lengths $(2a_1, 2a_2)$ and $(2b_1, 2b_2)$ respectively, is still a 
 parallelograms with sides  parallel to $(z_1, z_2)$, of lengths  $( 2a_ 1 ^ { 1 - \l} b _ 1 ^ \l, 2 a_ 2 ^ { 1 - \l}  b_ 2 ^ \l)$. 
\end{example}

\begin{example} [geometric combination of two rhombi]\label{e:rhombi}
 If $K = R ( a_1, a _ 2)$ is the centrally symmetric rhombus with one vertex at $(a_1, 0)$ and another one at $(0, a_2)$, with $a_ i >0$, we have $\mathcal H ^ 1 (K _ { u _ 1 ( s) }) = 2 a_2 \sqrt { 2 s}$,
  so that the i.d.~field of $K$ is given by 
 $$U ( t_1, t_2) = \big (   a _ 1 \big ( \sqrt { 2t_1} - 1 \big )   \, , \  a_2 \sqrt { 2 t _ 1} ( 2 t _ 2 -1)  \big ) \qquad \forall (t_1, t_2) \in (0, 1) ^ 2 \,.$$ 
If we take another rhombus $L = R ( b_1, b_2)$ with i.d.~field $V$, and we choose for instance $\lambda = \frac{1}{2}$, 
 the standard $(e_1, e_2)$-geometric potential of $(U, V)$ in proportion $\frac{1}{2}$  has components  
 $$W_{1/2} ( t_1, t _2) = \big (   \sqrt {a_1  b_1} ( \sqrt{ 2 t _ 1} - 1)  \, , \  \sqrt{b_1 b_2}  \sqrt{ 2 t _ 1} ( 2 t _ 2 -1)  \big )  \,, \qquad \forall (t_1, t_2) \in (0, 1) ^ 2 \,. $$ 
Hence, $$R ( a_1, a_2) \star_{\frac{1}{2} } R ( b_1, b_2) = R ( \sqrt {a_1 a_2}, \sqrt {b_1 b_2})\,.$$  
 \end{example}

\begin{example} [geometric combination of a rectangle and a rhombus] Using the same notation as in Examples \ref{e:rectangles} and \ref{e:rhombi},
 if $K= Q ( a_1, a_2)$ and $L = R ( b_1, b_2)$, 
 the  standard $(e_1, e_2)$-geometric potential in proportion $1/2$ has components  
 $$W_{1/2} ( t_1, t _2) = \Big (   \frac{4}{3}2  ^ {\frac{1}{4}}   \sqrt{a_1 b_1} \big ( t _ 1 ^ {\frac{3}{4}}    -   \big ( \frac{1}{2} \big )  ^ {\frac{3}{4}}  \big )  , 2  ^ {\frac{1}{4}} 
 \sqrt { a_2 b_2} t_1 ^ { \frac{1}{4}} \,  \big ( 2 t _ 2 - 1 \big  )  \Big ) \qquad \forall (t_1, t_2) \in (0, 1) ^ 2   \,.  $$

 Taking e.g. $a _ 1 = a _2 = b _ 1= b _ 2= 1$, one quarter of the boundary of $Q ( 1, 1) \star_ {\frac{1}{2}} R ( 1, 1)$ is the curve
 $$ 2  ^ {\frac{1}{4}} \Big ( \frac{4}{3} \big ( t _ 1 ^ {\frac{3}{4}}    -   \big ( \frac{1}{2} \big )  ^ {\frac{3}{4}} \big ), - t _ 1 ^ {\frac{1}{4}}   \Big ), \qquad t _ 1 \in \big  (0, \frac{1}{2}\big ) \, ,  $$
 namely the graph of the function 
 $$x (y) = - \frac{2 \sqrt 2}{3} (y ^ 3 + 1)\, ,  \qquad y \in (-1, 0)\,,$$  
 see Figure \ref{fig:1}.

    \begin{figure}[h] 
    \includegraphics[height=5cm]{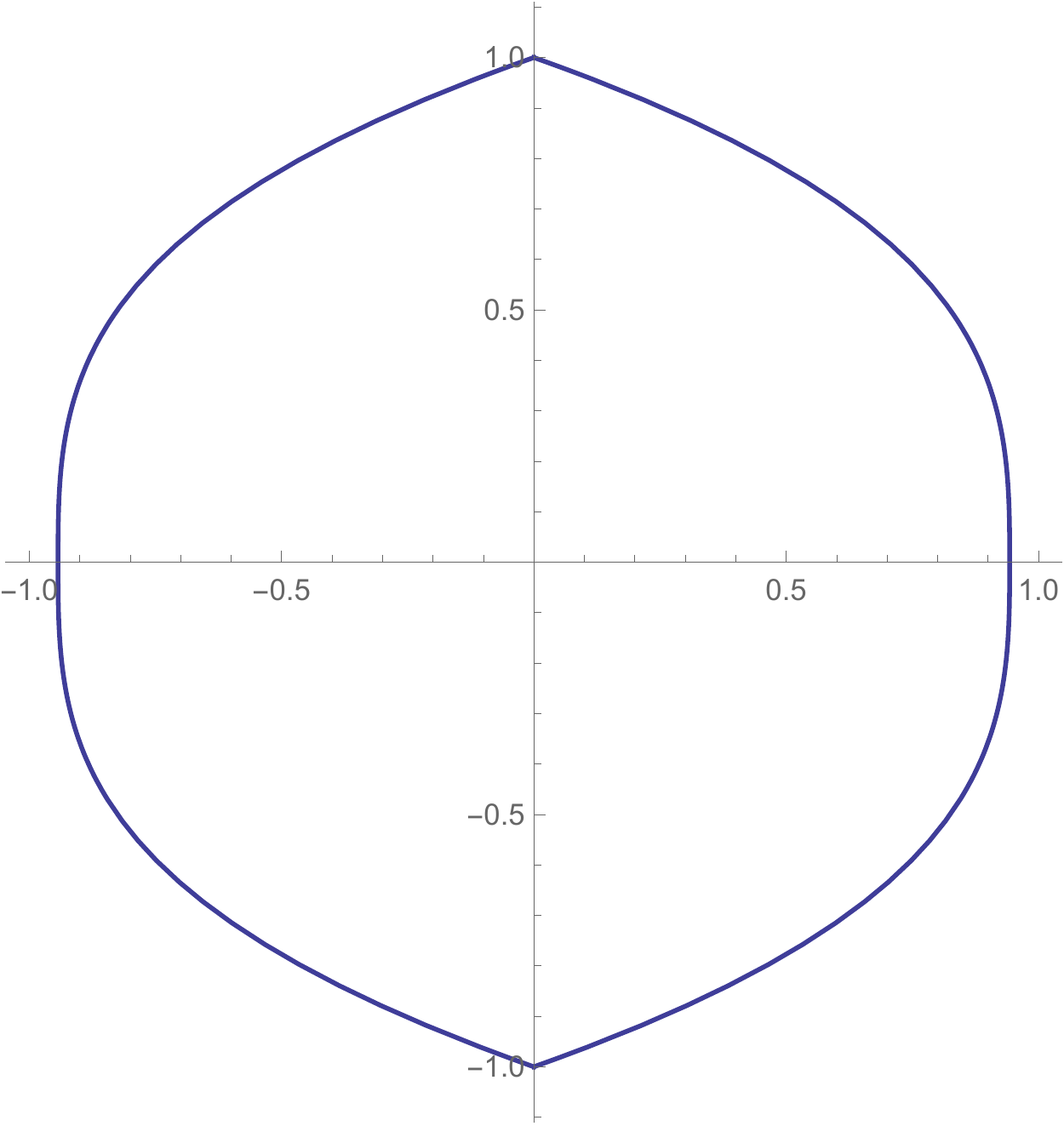}
\caption{The geometric combination $Q ( 1, 1) \star_ {\frac{1}{2}} R ( 1, 1)$ } 
\label{fig:1} %
\end{figure} 

 \end{example}

 \subsection{Unconditional convex bodies}\label{sec:unconditional} 
 
When $K$ is an unconditional convex body in $\R^n$, setting 
\[
K _ { x_1, \dots, x _i  } := \Big \{ ( x _ {i+1} e _ { i+1}+  \dots + x _n  e_n) \ :\ (x_1 e _ 1 + \dots +  x _ i e _ i +  x _{i+1} e _ { i+1} +  \dots + x _n e_n ) \in K \Big \}\,,
\] 
the components $(u_1, \dots, u _n)$ of the $(e_1, \dots, e_n)$-i.d.~field of $K$ satisfy 
\begin{itemize}
\item $t_1 \mapsto u _ 1(t_1)$ is the i.d.~function of the map
$x _ 1 \mapsto  
\mathcal H ^ {n-1}  ( K _ {x_1})$, 
hence
$$ u _ 1 ' ( t_1) = \frac{|K| }{\mathcal H ^ {n-1} ( K _ {u_1(t_1)})} \qquad \text{ for $\mathcal L ^ 1$-a.e. } t _ 1\in (0, 1) \,;$$
\item  for $\mathcal L ^ 1$-a.e.\ $t _1\in (0, 1)$, $t_2 \mapsto u _ 2 ( t_1 ,t_2)$ is the i.d.~function of the map 
$x _ 2 \mapsto  
\mathcal H ^ {n-2} ( K _ { (u _ 1 ( t_1 ), x_2)} )$;
hence,
\[
\frac{ \partial u _ 2}{\partial t _ 2}   (t_1, t _ 2)=  \frac{ \mathcal H ^ {n-1} ( K _ { u _ 1 ( t_1 )} )}{\mathcal H ^ {n-2} ( K _ { u _ 1 ( t_1 ), u_2(t_1, t_2)} )}  \qquad \text{ for $\mathcal L ^ 1$-a.e. } t _ 2\in (0, 1) \,; 
\] 
\item $\dots$
\item for $\mathcal L ^ {n-1}$-a.e. $(t _1, \dots, t _{n-1}) \in (0, 1) ^ { n-1}$, $t _n \mapsto u _n ( t _1, \dots, t _{n-1}, t _n)$ is the   i.d.~function of the map 
$x _n \mapsto 
\chi _K (    u _1 ( t_1) e _ 1 +  \dots + u _ {n-1} (t_{n-1}) e _ { n-1} +  x_n e _n)  $; hence,
\[
\frac{ \partial u _ n}{\partial t _ n}   (t_1, \dots, t _ {n-1}, t _n)=  \mathcal H ^ {1} ( K _ { u _ 1 ( t_1 ), \dots,  u _ {n-1} ( t _1, \dots, t _{n-1}  )}) \qquad \text{ for $\mathcal L ^ 1$-a.e. } t _ n\in (0, 1)  \,.
\] 
\end{itemize}

\medskip
\begin{remark}\label{r:interpolation}   
Given  an unconditional convex body  $K$, the maps
\[
s \mapsto  \mathcal H ^ { n-i} \left( K _ {u _ 1 ( t_1), \dots, u _ {i-1} ( t _1, \dots, t_{i-1}, s )   } \right ) \qquad i = 1, \dots, n\,,
\]
are continuous and strictly positive on the interior of their support (which is an interval).  
Hence, from their definition above, the components $(u_1, \dots, u _n)$ of the $(e_1, \dots, e_n)$-i.d.~field of $K$ 
satisfy
$$t _ i \mapsto u _ i ( t _1, \dots, t _ i) \in {\rm Lip}_{\rm loc} (0, 1) \qquad i = 1, \dots, n\,.$$ 
In view of Proposition \ref{p:contl} we infer that, given two  unconditional convex body  $K$ and $L$, the map $\lambda \mapsto K \star_\l L$ 
is a continuous interpolation in $L ^ 1$ from $K \star_0 L = K$ to  $K \star_1 L = L$. In addition, for $n = 2$, by Proposition \ref{p:convexity} this interpolation is
made of convex bodies. 
\end{remark}
 
\begin{proposition}\label{p:inclusion1} 
Let
$K$ and $L$ be unconditional convex bodies in $\R ^ n$, and let $\lambda \in [0,1]$.
Then the standard $(e_1, \dots, e _n)$-geometric combination $K \star_\l  L$  enjoys the following properties:
\begin{itemize}
\item[(i)] it is an unconditional set;
\item[(ii)] it satisfies  the inclusion
\begin{equation}\label{f:inclusion} K \star_\l  L \subseteq {( 1- \l)  \cdot K +_0 \l \cdot L }\,.
\end{equation}
\end{itemize} 
In particular, the equality  $| K \star_\l L | = |K| ^ { 1 - \l} |L| ^ \l$ given by Theorem \ref{t:nD}  implies the log-Brunn-Minkowski inequality for  unconditional convex bodies.
\end{proposition}

\begin{proof}   
(i) Let us
check,  for every fixed $i \in \{ 1, \dots, n\}$, the implication
\begin{equation}\label{f:symmetry} 
(x_1, \dots, x_i, \dots, x _n ) \in K \star_\l   L \ \Longrightarrow  \  (x_1, \dots, - x_i, \dots, x _n ) \in K \star_\l  L  \,.
\end{equation}
By construction, $K \star_\l  L$ agrees with the image of the  standard geometric potential $W_\l$, i.e., 
$$ K \star_\l  L = \Big \{  \big (w_{\l, 1} ( t_1), w _ {\l,2} ( t_1, t _ 2), \dots, w _{\l,n} ( t_1, \dots, t _n)  \big )   \ :\ t _ i \in (0,1)\Big \}\, .  $$ 
From the expressions of $\frac{\partial u _i}{\partial t _i}$ recalled at the beginning of Section \ref{sec:unconditional} we see that, since
$K$ is unconditional, the components $u _i$ of its i.d.~field satisfy 
\begin{equation}\label{f:even} 
\frac{\partial u _i}{\partial t _i} ( t_1, \dots,t_{i-1},  s) = \frac{\partial u _i}{\partial t _i} ( t_1, \dots,t_{i-1},  1-s) \qquad \text{ for } \mathcal L ^ 1\text{-a.e.}\  s \in (0,1) \,, 
\end{equation} 
and similarly for the components $v _i$ of the i.d.~field   of $L$. 

Then,  for every $i = 1, \dots, n$,  $\delta \in (0 , \frac{1}{2})$ and $t _ 1, \dots, t _ {i-1} \in (0,1)$,
it holds that
\[
\begin{array}{ll} w _{\l, i} ( t _1, \dots, t _ { i-1}, \frac {1}{2} + \delta )&  \displaystyle =  \int _{1/2} ^ { \frac {1}{2} + \delta } \frac{\partial u _i}{\partial t _i} ( t_1, \dots,t_{i-1},  s)  ^ {1- \l} \frac{\partial v _ i}{\partial t _i} ( t _ 1 , \dots, t _{i-1}, s)   ^ {\l}  \, ds  
\\  \noalign{\bigskip} &  \displaystyle =  \int _{\frac{1}{2}} ^ { \frac {1}{2} + \delta } \frac{\partial u _i}{\partial t _i} ( t_1, \dots,t_{i-1},  1-s)  ^ {1- \l} \frac{\partial v _ i}{\partial t _ {i}} ( t _ 1 , \dots, t _{i-1}, 1- s)   ^ {\l}  \, ds 
\\  \noalign{\bigskip} &  \displaystyle =  - \int _{\frac{1}{2}- \delta} ^ { \frac {1}{2}  } \frac{\partial u _i}{\partial t _i} ( t_1, \dots,t_{i-1},  s')  ^ {1- \l} \frac{\partial v _ i}{\partial t _ {i}} ( t _ 1 , \dots, t _{i-1}, s')   ^ {\l}  \, ds'
\\  \noalign{\bigskip} &=  - w _ {\l,i} ( t _1, \dots, t _ { i-1}, \frac {1}{2} - \delta ) \,,
\end{array} 
\]
where the first and fourth equalities hold by definition of $w _{\l,i}$, the third one by the change of variable $s' = 1-s$, and the second one is due to the fact that the functions  $u _i$ (and $v_i$) satisfy \eqref{f:even}. This shows
the implication \eqref{f:symmetry}.

\smallskip 
(ii) In order to show the inclusion \eqref{f:inclusion},  it is enough to prove that
\begin{equation}\label{f:inclusionprod} 
K \star_\l  L \subseteq K ^ { 1-\l} \cdot  L ^ \l \,,
\end{equation}
 where
 $$K ^ { 1 - \l} \cdot  L ^ \l:= \Big  \{ \big (  \pm  |x_1| ^ {1 - \l} |y_1| ^ {\l} ,  \dots , \pm  |x_n| ^ {1 - \l} |y_n| ^ {\l} \big ) \ :\ (x_1, \dots, x _n) \in K \, , \   (y_1, \dots, y _n) \in L \Big \}  \,.$$ 
Indeed, \eqref{f:inclusion} will follow from \eqref{f:inclusionprod} because 
$K ^ { 1-\l}  \cdot L ^ \l \subseteq ( 1- \l)\cdot K +_0 \l  \cdot L$ 
(see \cite[Lemma~4.1]{Sar}). 

Since we now  know that both sets $K \star_\l  L$ and $ K ^ { 1-\l} \cdot  L ^ \l$ are unconditional, we are reduced to show that
\begin{equation*}
(K \star_\l  L) \cap \R ^n _+  \subseteq \big ( K ^ { 1- \l} \cdot L ^ \l  \big ) \cap \R ^n _ + \,.
\end{equation*} 
Let us consider a  generic element of $(K \star_\l  L) \cap \R ^n _+$, which will be of the form
\begin{equation}\label{f:generic} 
W_\l (t):= \big( w _{\l, 1} ( t_1), w _ {\l,2} ( t_1, t_2), \dots, w _{\l,n} ( t _1, \dots , t _n)  \big ) 
\end{equation}  
for some $t = (t_1, \dots, t _n)$ with $t _i \in( {1}/{2}, 1)$ for every $i = 1, \dots, n$. 

By applying H\"older's inequality with conjugate exponents $\frac{1}{1 - \l}$ and $\frac{1}{\l}$, we see that
\begin{equation}\label{f:componenti} 0 \leq w _{\l, i} ( t _ 1, \dots , t _i) \leq u _ i ( t _ 1, \dots , t _i)  ^ { 1 - \l} v _ i ( t _ 1, \dots , t _i)  ^ {  \l}  \qquad \forall i = 1, \dots, n\,. 
\end{equation} 
Since the vector
$$\big( u _ 1 ( t_1) ^ {1 - \l} v _ 1 ( t_1) ^{ \l}  , u _ 2 ( t_1, t_2)^ {1 - \l} v _ 2 ( t_1, t_2)^ { \l} , \dots, u _n ( t _1, \dots , t _n)^ {1 - \l} v _n ( t _1, \dots , t _n)^ { \l}    \big ) $$ 
belongs to $ \big ( K ^ { 1- \l} \cdot L ^ \l  \big ) \cap \R ^n _ + $ and, since the sets $K$ and $L$  are convex (which implies $K ^ { 1 - \l} \cdot L ^ \l$ convex as well),  
for every $s = ( s_1, \dots, s _n) \in (0, 1)^n$ the vector 
$$ \big( s_1 u _ 1 ( t_1) ^ {1 - \l} v _ 1 ( t_1) ^{ \l}  , s_2 u _ 2 ( t_1, t_2)^ {1 - \l} v _ 2 ( t_1, t_2)^ { \l} , \dots, s_n u _n ( t _1, \dots , t _n)^ {1 - \l} v _n ( t _1, \dots , t _n)^ { \l}    \big ) $$ 
 belongs to $ \big ( K ^ { 1- \l} \cdot L ^ \l  \big ) \cap \R ^n _ + $. 
 Therefore,  the inequalities \eqref{f:componenti} imply  that  the vector $W_\l (t)$ in \eqref{f:generic} belongs to 
$  \big ( K ^ { 1- \l} \cdot L ^ \l  \big ) \cap \R ^n _ +$.
\end{proof}

\bigskip

\bigskip

 \subsection{Convex bodies with $n$ symmetries}\label{sec:symmetries} 

The class of convex bodies with $n$ symmetries has been 
considered in the literature  on Convex Geometry, in particular to give a partial answer to some long-standing open questions such as 
the Mahler conjecture \cite{barcor, barfra} and 
 the log-Brunn-Minkowski conjecture \cite{BoroKala}. 
 
Before giving the definition, let us recall first that a linear reflection is a map $A \in GL (n)$ which acts identically into some $(n-1)$-dimensional linear subspace $H$ of $\R ^n$, and there exists $u \in S ^ {n-1} \setminus H$ such that $A ( u) = - u$. In particular, 
an orthogonal reflection is a linear reflection $A$ which belongs to $O ( n)$.

Now, given a family  $A_1, \dots, A _n$ of linear reflections in $\R^n$,  such that the corresponding hyperplanes $H _ 1, \dots , H _n$ satisfy
$H _ 1 \cap \dots \cap H _n = \{ 0 \}$, we set
$$\mathcal Sym (A_1, \dots, A _n):= \Big \{ K \subseteq \R  ^n \ :\ A_i K = K \quad \forall i = 1, \dots, n \Big \}\,.$$

To deal with the operation of geometric combination in the class $\mathcal Sym(A_1, \dots, A _n)$, it is crucial to choose an appropriate family of directions. This requires to fix some background from \cite{BoroKala}:

\medskip

\begin{itemize} 
\item{}  If $(A_1, \dots, A _n)$ are orthogonal   reflections in $\R^n$,  such that the corresponding hyperplanes $H _ 1, \dots , H _n$ satisfy
$H _ 1 \cap \dots \cap H _n = \{ 0 \}$, we denote by  $C ( A_1, \dots, A _n)$ a {\it $n$-dimensional simplicial convex cone} (namely  the  
positive hull of $n$ linearly independent vectors) which is associated with the closure of the group generated by $(A_, \dots, A_n)$ as in \cite[Proposition 1]{BoroKala}.  If $C$ is the positive hull of $w_1, \dots, w_n$, the linear subspaces generated by $\{  w _ 1, \dots, w _n \} \setminus w _ i$, for $i = 1, \dots, n$, are called the {\it walls} of $C$.
\smallskip

\item{} If $(A_1, \dots, A _n)$ are merely linear   reflections in $\R^n$,  such that the corresponding hyperplanes $H _ 1, \dots , H _n$ satisfy
$H _ 1 \cap \dots \cap H _n = \{ 0 \}$, there  
exists a map  $\Psi \in GL (n)$ such that $A' _i:= \Psi A_i \Psi ^ { -1}$ are orthogonal reflections through hyperplanes
$H' _ 1, \dots , H' _n$ satisfying
$H' _ 1 \cap \dots \cap H' _n = \{ 0 \}$ ($\Psi$ can be found as a map which sends the L\"owner ellipsoid of $K$ into the unit ball, see the proof of \cite[Theorem 2]{BoroKala}). 
\end{itemize}

 \medskip
We refer to \cite{BoroKala} for more details, and we proceed to state the following

\begin{proposition}\label{p:inclusion2}
Let
$K$ and $L$ be convex bodies in the class $\mathcal Sym(A_1, \dots, A _n)$,   and let $\lambda \in [0,1]$.
Let $K \star_\l  L$ be their $(z_1, \dots, z_n)$-geometric combination in proportion $\lambda$, 
where the family $( z _1, \dots, z _n)$ is chosen as follows:
\begin{itemize}
\item[(a)]  if $A_i $ are orthogonal reflections, $(z_1, \dots, z_n)$ are the normals to the walls of the simplicial convex cone $C(A_1, \dots, A _n)$;  
\item[(b)] if $A_i $ are merely linear reflections, denoting by $\Psi$ a map in $GL (n)$ such that $A'
 _i:= \Psi A_i \Psi ^ { -1}$ are orthogonal reflections,  
 $(z_1, \dots, z_n)$ are the  image through $\Psi ^ {-1}$ of the normals to the walls of the simplicial convex cone $C(A
 '_1, \dots,A' _n)$. \end{itemize} 
Then $K \star_\l  L$ 
 enjoys the following properties:
\begin{itemize}
\item[(i)] it belongs to the class $\mathcal Sym(A_1, \dots, A _n)$;
\item[(ii)] it satisfies the inclusion 
\begin{equation}\label{f:inclusion2} K \star_\l  L \subseteq {( 1- \l)  \cdot K +_0 \l \cdot L }\,.
\end{equation}
\end{itemize} 
In particular, the equality  $| K \star_\l L | = |K| ^ { 1 - \l} |L| ^ \l$ given by Theorem \ref{t:nD}  implies the log-Brunn-Minkowski inequality in the class $\mathcal Sym(A_1, \dots, A _n)$
\end{proposition}

\begin{proof} 
Let us prove the statement 
first in case $A _i$ are orthogonal reflections and then in case they are merely linear reflections.

\smallskip
-- {\it Case (a). ($A _i$ orthogonal reflections)}. 

\smallskip
(i) In order to check that $K \star_\l  L$ belongs to the class $\mathcal Sym(A_1, \dots, A _n)$, let  us write  the sets $K$, $L$ and $K \star_\l L$ as
$$
\begin{array}{ll}
& K = \Big \{  u_ 1 ( t_1) z _ 1+  u _ 2 ( t_1, t _ 2) z _2 +  \dots + u _n ( t_1, \dots, t _n)z _n    \ :\ t _ i \in (0,1)\Big \}  \, , 
\\
\noalign{\bigskip} 
& L = \Big \{  v_ 1 ( t_1) z _ 1+  v _ 2 ( t_1, t _ 2) z _2 +  \dots + v _n ( t_1, \dots, t _n)z _n    \ :\ t _ i \in (0,1)\Big \}  \, , 
\\ 
\noalign{\bigskip}
& K \star_\l  L =  \Big \{  w_ {\l,1 } ( t_1) z _ 1+  w _ {\l,2} ( t_1, t _ 2) z _2 +  \dots + w _{\l,n}  ( t_1, \dots, t _n)z _n    \ :\ t _ i \in (0,1)\Big \}\,,
 \end{array}
$$
where
$U = u _ 1z_1+  \dots + u _n z _n $, $V= v_1 z _1 + \dots + v _n z _n$ are the $(z_1, \dots, z _n)$-inverse distribution fields of $K$ and $L$, 
and  $W_\l = w _ {\l,1} z _ 1 + \dots + w _{\l,n} z _n$ is the standard geometric potential   of $(U, V)$ in proportion $\lambda$. 

The fact that $K$ and $L$ belong to $\mathcal Sym (A_1, \dots, A _n)$ can be expressed as the system of equalities 
\begin{equation}\label{f:pari} \begin{array}{ll} 
& u _ i  ( t_1, \dots,t_{i-1},  s) =  -  u _i ( t_1, \dots,t_{i-1},  1-s) \qquad \forall s \in (0,1)\,,
\\ 
\noalign{\bigskip} 
& v _ i  ( t_1, \dots,t_{i-1},  s) =  -  v _i ( t_1, \dots,t_{i-1},  1-s) \qquad \forall s \in (0,1)\,.
 \end{array}
 \end{equation}
Then,  proceeding in the same way as in the first part of the proof of Proposition \ref{p:inclusion1}, we see that the functions $w_{\l,i}$ 
continue to satisfy the  analogous equalities:
\begin{equation}\label{f:pari2} w _{\l, i}   ( t_1, \dots,t_{i-1},  s) =  -  w _{\l,i}  ( t_1, \dots,t_{i-1},  1-s) \qquad \forall s \in (0,1)\,,\end{equation}
which implies that also $K \star_\l  L $ belongs to $\mathcal Sym (A_1, \dots, A _n)$. 

\smallskip
(ii) By statement (i) already proved, we know that $ K \star_\l  L$  belongs to $\mathcal Sym (A_1, \dots, A _n)$. As well, since the $0$-sum is linear covariant (i.e.\ $A_i(( 1- \l)  \cdot K +_0 \l \cdot L)) = ( 1- \l)  \cdot A_i (K ) +_0 \l \cdot A_i (L)$), we have that
 $( 1- \l)  \cdot K +_0 \l \cdot L$ belongs to $\mathcal Sym (A_1, \dots, A _n)$. 
Therefore, in order to prove the inclusion \eqref{f:inclusion2}, we are reduced to show that, denoting for brevity by $C$ the simplicial convex cone $C ( A _1, \dots, A _n)$, it holds that
\begin{equation}\label{f:inclusionC} C \cap  ( K \star_\l  L)  \subseteq  C \cap   
\big ({( 1- \l)  \cdot K +_0 \l \cdot L } \big )  \,.
\end{equation}
In turn, to have \eqref{f:inclusionC} it is enough to show that
\begin{equation*}
\Phi \big ( C \cap ( K \star_\l  L)  \big ) \subseteq \Phi \big ( C \cap \big (  {( 1- \l)  \cdot K +_0 \l \cdot L } \big ) \big ) \,,
\end{equation*} 
where $\Phi$ is a map in $GL (n)$ with $\Phi ( z _ i) = e _ i$, so that $\Phi$ maps $C$ into $\R ^ n _+$
(recall that $( z_1, \dots, z _n)$ denote the normals to the walls of  $C$).

Since $K$ and $L$ are invariant under the closure of the group generated by $(A_1, \dots, A_n)$,  
by \cite[Proposition 1 (v)]{BoroKala} we know that
the unconditional sets $\overline K$  and $\overline L$ defined  by 
$$ \R ^n _+ \cap  \overline K  := \Phi ( C \cap K) \qquad \text{ and } \qquad 
\R ^n _+ \cap \overline L  := \Phi (  C \cap L) $$
are unconditional convex bodies.

Moreover, from the proof of Theorem 8 and Lemma  6 (ii) in \cite{BoroKala}, we know that 
$$\begin{array}{ll}
 \displaystyle \R ^ n _+ \cap \big (  {( 1- \l)  \cdot \overline K +_0 \l \cdot \overline  L } \big ) &  \displaystyle \subseteq \Phi 
 \big ( \big \{ x \in C \ :\ \langle x, u \rangle \leq h _K ( u) ^ { 1 - \l} h _ L ( u) ^ \l  \ \forall u \in C  \big \} \big ) 
 \\ \noalign{\medskip} &  \displaystyle = \Phi \big ( C \cap \big (  {( 1- \l)  \cdot K +_0 \l \cdot L } \big ) \big )   \,.
 \end{array}
$$  
We are thus reduced to prove that
\begin{equation}\label{f:uncon} 
\Phi \big ( C \cap ( K \star_\l  L)  \big ) \subseteq \R ^ n _+ \cap \big (  {( 1- \l)  \cdot \overline K +_0 \l \cdot \overline  L } \big )\,.
\end{equation} The inclusion \eqref{f:uncon} will follow from Proposition  \ref{p:inclusion1} applied to the unconditional convex bodies $\overline K$ and $\overline L$,  provided we are able to show that 
\begin{equation}\label{f:finale} 
\Phi \big ( C \cap ( K \star_\l  L)  \big )= \R ^ n _+  \cap  (\overline K \star_\l \overline L ) \,, 
\end{equation} 
where $\overline K \star_\l \overline L$ is the $(e_1, \dots , e_n)$-geometric combination  of $\overline K$ and $\overline L$.
 
We emphasize that the families of vectors with respect to which the two geometric combinations appearing in \eqref{f:finale} are constructed are distinguished, and their indication is omitted just for notational simplicity.  For the sake of clearness, let us
repeat that $K \star_\l  L$  
is the $(z_1, \dots , z_n)$-geometric combination of $K$ and $L$, with  $(z_1, \dots, z_n)$  chosen as specified in the statement of Proposition \ref{p:inclusion2}, while $\overline K \star_\l \overline L$ is the $(e_1, \dots , e_n)$-geometric combination  of the unconditional bodies $\overline K$ and $\overline L$.

Let us check the equality \eqref{f:finale}.  By the system of equations \eqref{f:pari}-\eqref{f:pari2}, we see that 
$$u _ i  ( t_1, \dots,t_{i-1},  1/2 ) = v _ i  ( t_1, \dots,t_{i-1},  1/2 )  = w_{\l,  i}  ( t_1, \dots,t_{i-1},  1/2 )  =  0 \quad \forall i = 1, \dots n\,.$$ 
Hence,  the intersections of  the walls of $C$ with $K$, $L$ and $K \star _\l L$ are given respectively by
 $$K \cap \Big \{ t _ i = \frac{1}{2 } \Big \} \, ,  \quad L \cap \Big \{ t _ i = \frac{1}{2} \Big \}\, , \quad (K \star _\l L)  \cap\Big \{ t _ i = \frac{1}{2 }\Big \}\,.$$ 
Thus, 
$$\begin{array}{ll}
&C \cap K = \Big \{  u_ 1 ( t_1) z _ 1+  u _ 2 ( t_1, t _ 2) z _2 +  \dots + u _n ( t_1, \dots, t _n)z _n    \ :\ t _ i \in (1/2,1)\Big \}  \, , 
\\
\noalign{\bigskip} 
&C \cap  L = \Big \{  v_ 1 ( t_1) z _ 1+  v _ 2 ( t_1, t _ 2) z _2 +  \dots + v _n ( t_1, \dots, t _n)z _n    \ :\ t _ i \in (1/2,1)\Big \}  \, , 
\\ 
\noalign{\bigskip}
& C \cap ( K \star_\l  L) =  \Big \{  w_ {\l,1} ( t_1) z _ 1+  w _ {\l,2} ( t_1, t _ 2) z _2 +  \dots + w _{l,n} ( t_1, \dots, t _n)z _n    \ :\ t _ i \in (1/2,1)\Big \}\,.
 \end{array}
$$ 
By applying the map $\Phi$ to the last equality above, we see that the set at the left-hand side of \eqref{f:finale} satisfies
\begin{equation}\label{f:prima}  \Phi ( C \cap ( K \star_\l  L ) ) 
   \!\!= \!\!\Big \{ w_ {\l,1} ( t_1) e _ 1+  w _ {\l, 2} ( t_1, t _ 2) e _2 +  \dots + w _{\l,n} ( t_1, \dots, t _n)e _n    \, :\, t _ i \in (1/2,1) \Big \}\,.
   \end{equation} 

On the other hand, we have 
 $$\begin{array}{ll}  \R^ n _+ \cap \overline K   & =   \Phi (C \cap  K ) = \Phi\Big (  \Big \{  u_ 1 ( t_1) z _ 1+  u _ 2 ( t_1, t _ 2) z _2 +  \dots + u _n ( t_1, \dots, t _n)z _n    \ :\ t _ i \in ({1}/{2},1)\Big \}   
 \Big ) \\ \noalign{\bigskip} & =   \Big \{  u_ 1 ( t_1) e _ 1+  u _ 2 ( t_1, t _ 2) e _2 +  \dots + u _n ( t_1, \dots, t _n)e _n    \ :\ t _ i \in (1/2,1)\Big \}   
 \end{array}$$ 
 and similarly for $L$. 
 In view of the equalities \eqref{f:pari}, we infer that 
 $$\begin{array}{ll} 
 & \overline K = \Big \{  u_ 1 ( t_1) e _ 1+  u _ 2 ( t_1, t _ 2) e _2 +  \dots + u _n ( t_1, \dots, t _n)e _n    \ :\ t _ i \in (0,1)\Big \} \,,   
 \\  \noalign{\medskip} 
 & \overline L = \Big \{  v_ 1 ( t_1) e _ 1+  v _ 2 ( t_1, t _ 2) e _2 +  \dots + v _n ( t_1, \dots, t _n)e _n    \ :\ t _ i \in (0,1)\Big \}   \,. 
 \end{array}
 $$ 
It follows straightforwardly that 
 $\overline U := u _ 1 ( t_1) e _ 1+   u _ 2 ( t_1, t _ 2) e _2 +  \dots + u _n ( t_1, \dots, t _n)e _n $ and 
 $\overline V := v_ 1 ( t_1) e _ 1+  v _ 2 ( t_1, t _ 2) e _2 +  \dots + v _n ( t_1, \dots, t _n)e _n $ are respectively the $( e _1, \dots, e _n)$-i.d.~fields of the unconditional bodies $\overline K$ and $\overline L$. 
Therefore, 
  $$ \overline K \star _{\l} \overline L = \Big \{  w_ {\l,1} ( t_1) e _ 1+  w _ {\l,2} ( t_1, t _ 2) e _2 +  \dots + w _{\l,n} ( t_1, \dots, t _n)e _n    \ :\ t _ i \in (0,1)\Big \}\,.\\
    $$  
Hence, the set at the right-hand side of \eqref{f:finale} satisfies
\begin{equation}\label{f:seconda}  \R ^ n _+  \cap  (\overline K \star_\l \overline L ) 
   = \Big \{  w_ {\l,1} ( t_1) e _ 1+  w _ {\l,2} ( t_1, t _ 2) e _2 +  \dots + w _{\l,n} ( t_1, \dots, t _n)e _n    \ :\ t _ i \in (1/2,1)\Big \}\,.
   \end{equation} 
 By combining \eqref{f:prima} and \eqref{f:seconda}, equality \eqref{f:finale} follows.

\smallskip

 \bigskip
 -- {\it Case (b) ($A _i$ linear reflections)}. Let $\Psi$ be a map in $GL (n)$ such that $A' _i:= \Psi A_i \Psi ^ { -1}$ are orthogonal reflections. 
Let $\eta _i$ be the normals to the walls of the simplicial convex cone $C(A'_1, \dots, A' _n)$, so that  
$(z_1, \dots, z_n)  := \Psi ^ {-1} (\eta _1, \dots, \eta _n)$. 
Statements (i)-(ii)  readily follow from the corresponding items already proved in Case (a) provided we  
 show that 
  \begin{equation}\label{f:Psi} \Psi (K \star_\l  L) = \Psi (K) \star_\l \Psi ( L)\,, \end{equation}  
  where the geometric combination $K \star_\l  L$ is made with respect to $(z_1, \dots, z_n)$, and the geometric combination  $\Psi (K) \star_\l \Psi ( L)$ is made with respect to  $(\eta_1, \dots, \eta _n)$.  
 
 To prove \eqref{f:Psi}, we start by writing  $K \star_\l  L$ as
\[
K \star_\l  L =  \Big \{  w_ {\l,1} ( t_1) z _ 1+  w _ {\l,2} ( t_1, t _ 2) z _2 +  \dots + w _{\l,n} ( t_1, \dots, t _n)z _n    \ :\ t _ i \in (0,1)\Big \}\,,
\] 
 where $w _{\l, i}$ are the components of the standard $(z_1, \dots, z _n)$-geometric potential of the 
 $(z_1, \dots, z _n)$-i.d.~fields 
 of $K$ and $L$. 
 
Then
$$\Psi(K \star_\l  L) =  \Big \{  w_ {\l,1} ( t_1) \eta _ 1+  w _ {\l,2} ( t_1, t _ 2) \eta _2 +  \dots + w _{\l,n} ( t_1, \dots, t _n)\eta _n    \ :\ t _ i \in (0,1)\Big \}\,,\\
    $$
 so that \eqref{f:Psi} holds true provided that $w _{\l, i}$ are also  the components of the standard $(\eta_1, \dots, \eta _n)$-geometric potential of the 
 $(\eta_1, \dots, \eta _n)$-i.d.~fields  of $\Psi ( K)$ and $\Psi ( L)$. 
 In turn, this is true provided the following implication holds: if 
 $U (t)= u _ 1 ( t_1) z _ 1+   u _ 2 ( t_1, t _ 2) z _2 +  \dots + u _n ( t_1, \dots, t _n)z _n$ is the   
 $(z_1, \dots, z _n)$-i.d.~field of $K$, then $\widetilde U (t)= u _ 1 ( t_1) \eta _ 1+   u _ 2 ( t_1, t _ 2) \eta _2 +  \dots + u _n ( t_1, \dots, t _n)\eta  _n$ is the   
 $(\eta_1, \dots, \eta _n)$-i.d.~field of $\Psi(K)$  (and similarly for $L$).  Such implication follows immediately from Definition \ref{d:field}.
\end{proof}

\bigskip
\section{Open problems} \label{sec:open}

\begin{itemize}
\item{}
\textit{Problem 4.1: Continuous interpolations.} 
Given two  functions
$f,g\colon\R^n\to\R_+$, with strictly positive integrals and i.d.~fields $U, V$ respectively,  
construct a family of fields  $\{ W _\lambda \} _{\lambda \in [0, 1]}$ 
such that,  if $f \star _\l g$ is defined according to \eqref{f:nconv}, there holds:  

\begin{itemize}

\item[(i)]  For $\lambda = 0$ and $\lambda = 1$, we have $f \star _0 g = f$ and $f \star _ 1 g = g$. 
\smallskip 

\item[(ii)] The map $[0, 1] \ni \lambda \mapsto f \star _\l g $ is continuous  in $L ^ 1 ( \R^n)$. 
\smallskip

\item[(iii)]  For every $\l \in [0, 1]$,  $\displaystyle \int _{\R^n} f  \star_\l g  ( x) \, dx  = \Big (\int _{\R ^n} f ( x) \, dx \Big )^ { 1 - \l}  \Big (\int _{\R^n} g ( x) \, dx  \Big )^ {\l}.$

\end{itemize} 

\smallskip 
\noindent In this respect we have seen that, in dimension $n=1$, taking  $w _\l$ equal to a geometric primitive of $(u,v)$ in proportion $\lambda$, property (iii) is always satisfied by Theorem \ref{t:1D}, whereas properties (i)-(ii) are satisfied in case  $u, v$ belong to $AC_{\rm loc} (0, 1)$,  but they are not in case the distributional derivatives of $u$ or $v$
contain jumps or Cantor parts
 (cf.~Remark \ref{r:ac} and Proposition \ref{p:contl}). Moreover, we recall that (i)-(ii)-(iii) hold in any space dimension when $f$ and $g$ are the characteristic functions of two unconditional convex bodies, see Remark \ref{r:interpolation}.

\bigskip 
\item{} {\it Problem 4.2: Convexity preserving of geometric combination in dimension $n >2$.} 
Establish whether  (at least under the additional assumption that $K$ and $L$ are unconditional), Proposition \ref{p:convexity} continues to hold in dimension $n>2$.  
\bigskip 
\item{} {\it Problem 4.3: Comparison between geometric combination and $0$-sum.}  Establish whether (at least in dimension $n=2$) the inclusion \eqref{f:inclusion} in Proposition \ref{p:inclusion1} continues to hold for  arbitrary convex bodies (not necessarily unconditional), provided that the family of directions
$(z_1, \dots, z_n)$ needed to construct  $K \star_\l L$
 is suitably chosen.

\end{itemize}

\def\cprime{$'$}

\end{document}